\documentclass[11pt]{amsart}

\usepackage{amssymb,amsmath}
\usepackage{graphicx}
\usepackage{epsfig}
\usepackage{psfrag}
\usepackage{color}
\usepackage{natbib}
\usepackage{hyperref}

\newcommand{\gauss}[2]{exp( -((x-#1)^2)/(#2)^2)/sqrt(2*pi*(#2)^2)}

\newcommand{\expE}{\mathbb{E}}

\newcommand\ol[1]{{\overline {#1}}}

\newcommand\norm[1]{\lVert #1 \rVert}

\newcommand{\excise}[1]{}

\newtheorem{thm}{Theorem}[section]
\newtheorem{prop}[thm]{Proposition}

\newtheorem{lem}[thm]{Lemma}
\newtheorem{cor}[thm]{Corollary}

\theoremstyle{definition}
\newtheorem{defn}[thm]{Definition}
\newtheorem{rmk}[thm]{Remark}
\newtheorem{ex}[thm]{Example}
\newtheorem{conv}[thm]{Assumption}

\newcommand\Rm{\mathbb R}
\newcommand\Zm{\mathbb Z}

\newcommand\Em{\mathbb E}
\newcommand\Pm{\mathbb P}

\newcommand\Ical{\mathcal{I}}
\newcommand\Mcal{\mathcal{M}}
\newcommand\Scal{\mathcal{S}}
\newcommand\Tcal{\mathcal{T}}

\newcommand\0{\mathbf{0}}
\newcommand\KK{\mathcal{K}}
\newcommand\RR{\Rm}
\newcommand\ZZ{\Zm}

\newcommand\no{\nonumber}

\renewcommand\implies{\Rightarrow}

\DeclareMathOperator*{\argmin}{arg\,min}
\DeclareMathOperator*{\argmax}{arg\,max}

\begin{document}

\title{Sticky central limit theorems at\linebreak\
	isolated hyperbolic planar singularities}
\author{Stephan Huckemann, Jonathan C. Mattingly, Ezra Miller, and James Nolen}
\address{SH: Felix Bernstein Institute for Mathematical Statistics in
	the Biosciences\\University of
	G\"ottingen\\Goldschmidtstrasse~7, 37077 G\"ottingen, Germany}
\email{huckeman@math.uni-goettingen.de}
\address{JCM: Mathematics Department and Department of Statistical
	Science\\Duke University\\Durham, NC 27708, USA}
\email{jonm@math.duke.edu}
\address{EM: Mathematics Department\\Duke University\\Durham, NC 27708, USA}
\email{www.math.duke.edu/\~{\hspace{-.3ex}}ezra}
\address{JN: Mathematics Department\\Duke University\\Durham, NC 27708, USA}
\email{nolen@math.duke.edu}

\begin{abstract}
We derive the limiting distribution of the barycenter $b_n$ of an
i.i.d.\ sample of $n$ random points on a planar cone with angular
spread larger than $2\pi$.  There are three mutually exclusive
possibilities: (i)~(\emph{fully sticky} case) after a finite random
time the barycenter is almost surely at the origin; (ii)~(\emph{partly
sticky} case) the limiting distribution of $\sqrt{n} b_n$ comprises a
point mass at the origin, an open sector of a Gaussian, and the
projection of a Gaussian to the sector's bounding rays; or
(iii)~(\emph{nonsticky} case) the barycenter stays away from the
origin and the renormalized fluctuations have a fully supported limit
distribution---usually Gaussian but not always.  We conclude with an
alternative, topological definition of stickiness that generalizes
readily to measures on general metric spaces.
\end{abstract}

\date{3 July 2015}


\maketitle

\setcounter{tocdepth}{3}
\tableofcontents

\section*{Introduction}\label{s:intro}

It has recently been observed that large samples from well-behaved
probability distributions on metric spaces that are not smooth
Riemannian manifolds are sometimes constrained to lie in subsets of
low dimension, and that central limit theorems in such cases
consequently behave non-classically, with components of limiting
distributions supported on thin subsets of the sample space
\cite{HHMMN13,BardenLeOwen2013,Basrak2010}.  Our results here continue
this line of investigation with the first complete description of
``sticky'' behavior at a singularity of codimension~$2$.

More precisely, we prove laws of large numbers
(Theorem~\ref{t:LLN-intro}; see Section~\ref{s:slln} for proofs and
more details) as well as central limit theorems
(Section~\ref{sec:central-limit-theorem}; proofs in
Section~\ref{s:clt1}) for Fr\'echet means of probability distributions
(Definitions~\ref{d:populationbarycenter} and~\ref{d:moments}) on
metric spaces possessing the simplest geometric singularities in
codimension~$2$.  The spaces are surfaces homeomorphic to the
Euclidean plane~$\RR^2$ and metrically flat locally everywhere except
at a single \emph{cone point}, where the angle sum---the length of a
circle of radius~$1$---exceeds~$2\pi$ (see Section~\ref{sub:isolated}
for precise definitions).  Thus the surface is planar, the singularity
is isolated, and its geometry is hyperbolic, in the sense of
negatively curved; hence the title of this paper.

The asymptotic behavior splits into three cases, called \emph{fully
sticky}, \emph{partly sticky}, and \emph{nonsticky}
(Definition~\ref{d:stickyclassification} and
Proposition~\ref{p:casesSimple}), according to whether the mean lies
stably at the singularity (Theorem~\ref{t:fullstickyclt}), unstably at
the singularity (Theorem~\ref{t:stickyclt}), or away from the
singularity (Theorem~\ref{t:non-stickyclt}), respectively.  Specific
examples illustrating the sticky phenomena, including subtle non-local
effects of the singular negative curvature when the mean lies in the
smooth stratum (Example~\ref{e:Gaussian}), occupy
Section~\ref{s:examples}.  In contrast to the usual strong law
asserting almost-sure convergence of empirical means to a population
mean, sticky strong laws deal also with the limiting behavior of
supports of the laws of empirical means.  In the sticky case this
support degenerates in some specified sense already in finite random
time (Theorem~\ref{t:LLN-intro}).  Our sticky central limit theorems
assert that the limiting distributions are mixtures of parts of
Gaussians and collapsed (i.e., projected) parts of Gaussians.  Even in
the nonsticky case, the limiting laws can fail to be Gaussian
(Example~\ref{e:Gaussian}), which may come as a surprise: although the
space is locally Euclidean near the mean, the conclusion of
Theorem~2.3 of Bhattacharya and Patrangenaru (2005) can nonetheless
not be valid.

Concluding our analysis is a topological characterization of
stickiness for measures on isolated planar hyperbolic singularities
(Theorem~\ref{t:TopSticky}), as opposed to the algebraic one in terms
of moments (Definitions~\ref{d:moments}
and~\ref{d:stickyclassification}) used for the rest of the exposition.
Thinking topologically leads to a very general notion of stickiness
(Definition~\ref{d:sticky}), which we include with an eye toward
sampling from more general geometrically or topologically singular
spaces.  We have in mind stratified spaces (see [\cite{GM1988}] or
[\cite{Pflaum2001}]), suitably metrized, noting that (for example)
every real semialgebraic variety admits a canonical Whitney
stratification with finitely many semialgebraic
strata \mbox{\cite[Section~2.7]{GWdPL1976}}.

A motivating example of such stratified sample spaces comes from
evolutionary biology, where the objects are phylogenetic trees.  The
space of such objects is CAT(0) (or equivalently, globally
nonpositively curved) [\cite{BilleraHolmesVogtmann2001}] and therefore
has many desirable features where geometric probability is concerned
[\cite{Sturm2003}].  \cite{BardenLeOwen2013} treat the space~$\Tcal_4$
of phylogenetic trees with four leaves.%
	\footnote{As this draft was completed, the preprint
	\cite{BardenLeOwen2014} was posted.  The results there are
	proved for arbitrary numbers of leaves but restrict to
	singularities in codimension~$0$ and~$1$.}
The singularity of~$\Tcal_4$ at its center cone point is a
(non-disjoint) union of a certain number of copies of an isolated
planar singularity with angle sum~$5\pi/2 > 2\pi$.  Therefore some
features of our results are present in the central limit theorem at
the cone point of~$\Tcal_4$ \cite[Theorem~5.2]{BardenLeOwen2013},
which identifies the support of the limit measure in each right-angled
orthant as a cone over an interval.  However, the limit measure
exhibits additionally non-classical behavior at the boundary of its
support, where mass concentrates on the edges and even more on the
origin.  The simpler nature of an isolated planar singularity, which
lacks the global combinatorial complexity of tree space, allows us to
discover these boundary components and characterize them by
identifying the limit measure as the convex projection of a Gaussian
distribution (Theorem~\ref{t:stickyclt}).

While the strong law of large numbers on quasi-metric spaces by
\cite{Z77} and on manifolds by \cite{BP03} requires the existence of a
population mean, meaning square-integrability of the underlying law,
for fully sticky strong laws the existence of a population mean is not
necessary: no square-integrability is required.  Curiously, a (fully)
sticky central limit theorem can consequently hold in the absence of
any population mean at all (Example~\ref{ex:no-mean}).  That said, the
greater challenge consists in the partly sticky case; as is the case
for the multivariate Central Limit Theorem, as well as for that on
manifolds by \cite{HL96,HL98,BP05} or on certain stratified spaces by
\cite{H_Procrustes_10}, square-integrability is still required.

In addition to the theoretical interest in the asymptotic behavior of
means on stratified spaces, another driving motivation comes from the
need to accordingly devise inferential statistical methods for
applications based on the asymptotic behavior of Fr\'echet sample
means and similar mean quantities,
e.g.~\cite{Holmes2003,AydinPatakiWangBullitMarron2009,Nye2011,SBHMOOPPM14}.
This type of development is exemplified, in the form of confidence
intervals on the spider, by \cite{HotzLe2014}.

Many parts of this paper are rather technical---though
elementary---and require the buildup of notation in
Sections~\ref{s:hyperbolic} and~\ref{s:moments}, as we fold the
isolated planar singularity onto~$\RR^2$.  The behavior of first
moments under folding and rotation is essential to understand the
limiting location of barycenters on the singular space~$\KK$ (which we
call the \emph{kale}), and their limiting laws on~$\KK$ as well as
on~$\RR^2$, which are described by certain sectors where a first
folded moment is non-negative.  A list of notion is given in Section \ref{sec:notation}.
%

%

%

\subsection*{Acknowledgements}

The authors acknowledge support through the Niedersachsen Vorab of the
Volkswagen Foundation (SH) as well as grants from the US National
Science Foundation: DMS-0854879 (JCM); DMS-1001437 (EM); DMS-1007572, DMS-1351653
(JN); and the Statistical and Applied Mathematical Sciences Institute,
SAMSI, DMS-1127914~(SH).

\section{Basic definitions and principal results}\label{s:results}

\subsection{Isolated hyperbolic planar singularities}\label{sub:isolated}

The \emph{kale} is the space
\begin{equation}\label{eq:kale}
  \KK = \big((0,\infty) \times (\Rm/\alpha\Zm)\big)\cup \{\0\}
\end{equation}
where $\alpha > 2\pi$ is the \emph{angle sum} at the isolated point
$\0$, called the \emph{origin}, the sole point at which the metric is
not locally Euclidean.  Points are specified by polar coordinates $p =
(r,\theta) \in \KK$ for a \emph{radius} $r > 0$ and \emph{angle}
$\theta \in \Rm/\alpha\Zm$, and the origin is often expressed as $\0 =
(0,0)$ or $\0 = (0,\theta)$ for any $\theta \in \Rm/\alpha\Zm$; that
is, the origin is viewed as lying at zero radius along every ray
emanating from it.
The circle $\Rm/\alpha\Zm$, a group under addition, has the natural
uniform metric defined by
$$%
  |\theta'-\theta|_\alpha
  =
  \min_{n \in \Zm} |n \alpha + \theta' - \theta|.
$$
Note that $|\theta'-\theta|_\alpha\leq \alpha/2$.  Denote by
$d(p_1,p_2)$ the metric on $\KK$ defined by
$$%
d\big((r_1,\theta_1),(r_2,\theta_2)\big)^2 =
  \begin{cases}
    (r_1 + r_2)^2 &\text{ if } |\theta_1-\theta_2|_\alpha \geq \pi,
    \\r_1^2+r_2^2-2r_1r_2\cos\big(|\theta_1-\theta_2|_\alpha\big)
    &\text{ if } |\theta_1-\theta_2|_\alpha \leq \pi.
  \end{cases}
$$
When one of the points is the origin---so one of the radii
vanishes---both cases apply, and in that situation the distance equals
the other radius.  Geometrically, $\KK$ is the metric cone over a
circle of length $\alpha$ placed at distance~$1$ from the cone
point~$\0$.

If we allowed $\alpha = 2\pi$, then this construction would yield $\KK
= \Rm^2$ with the Euclidean metric.  If we allowed $\alpha < 2\pi$,
then this construction would be a right circular (``ice cream'') cone
with angle sum $\alpha$ at the apex.  The cases where the angle sum
$\alpha$ is bigger than, equal to, or smaller than $2\pi$ correspond
to the curvature at the origin being negative, flat, or positive,
respectively.  The name ``kale'' derives from the negative curvature
of that particular leafy~vegetable.

\begin{defn}\label{d:interval}
From now on, write $|\theta' - \theta|$ for $|\theta' -
\theta|_\alpha$, the role of $\alpha$ being understood.  When
$\big|\theta' - \theta\big| \leq \pi$, we identify $\theta' - \theta$
with a number in the closed interval $[-\pi,\pi]$. Specifically, there
is a unique integer $n$ such that $\theta' - \theta + n \alpha \in
[-\pi,\pi]$, and in this case we set
$$%
  \theta' - \theta = \theta' - \theta +  n \alpha \in [-\pi,\pi].
$$
\end{defn}

Definition~\ref{d:interval} implies that when $|\theta' - \theta|\leq
\pi$, the intervals of length $|\theta' - \theta|$ with endpoints
$\theta$ and~$\theta'$, closed or open at either end, are all well
defined in $\RR/\alpha\ZZ$.  In fact, even the interval $[\theta -
\pi,\theta + \pi] \subset \RR/\alpha\ZZ$ is well defined for all
$\theta\in \RR/\alpha\ZZ$, because $\alpha > 2\pi$.  If $|\theta -
\theta'| \leq \pi$, then the intervals $[\theta, \theta'] =
[\theta',\theta]$ coincide as subsets of $\RR/\alpha\ZZ$; it matters
not whether $\theta -\theta' < 0$ or~$\theta - \theta' > 0$.

\begin{defn}\label{d:sector}
If $I \subset \Rm/\alpha\Zm$ is any interval of angles, define the
\emph{sector}
$$%
  C_{I} = \left \{ (r,\theta) \in \KK \mid r \geq 0 \text{ and } \theta
  \in I \right\}
$$
that is the cone over~$I$ from the origin.  (If $I$ is closed, then
$C_I$ is a closed subset of $\KK$.)
\end{defn}
\begin{excise}{%
  \begin{eqnarray*}*\label{minimal-n}
  \big|\theta'-\theta\big|_\alpha &:=& | \alpha n^*\theta'-\theta|\mbox{
  with any } n^*\in \argmin_{n\in\mathbb Z} |\alpha n\theta'-\theta|\,.
  \end{eqnarray*}
}\end{excise}%

\begin{defn}\label{d:folding}
For a fixed angle $\theta \in \Rm/\alpha\Zm$, the \emph{folding map} $F_\theta\colon \KK \to \Rm^2$ is
determined by
$$%
F_\theta(r',\theta') =
  \begin{cases}
                0           & \text{if } r' = 0
   \\(r', \theta' - \theta) & \text{if } r' > 0\text{ and }|\theta' - \theta|\leq\pi
   \\       (r',\pi)        & \text{if } r' > 0\text{ and }|\theta' - \theta|\geq\pi.
  \end{cases}
$$
\end{defn}
Here we are using polar coordinates for both $\KK$ and $\Rm^2$; later we will sometimes use cartesian coordinates for the image of $F_\theta$.  Observe that when $|\theta' - \theta| = \pi$ the second and third cases agree.  A
simple geometric description of the folding map is given in terms of
light and shadow as follows, cf.~also Figure~\ref{f:folding-map}.

\begin{defn}\label{d:IcalA}
The open set
$$%
  \Ical_\theta = \big\{(r,\theta') \in \KK \mid r > 0 \text{ and }
  |\theta' - \theta| > \pi\big\} \subset \KK
$$
is the part of~$\KK$ \emph{invisible from} the angle~$\theta$.  The
complement $\KK\setminus \Ical_{\theta}$ is the part \emph{visible
from~$\theta$}.  The complement $\KK \setminus
\ol{\Ical_\theta}$ of the closure of the invisible part is
\emph{fully visible}, and the set $\ol{\Ical_\theta} \setminus
\Ical_\theta$ of boundary points outside of~$\Ical_\theta$ is
\emph{partly visible}.  The \emph{shadow} of any set $A \subseteq
\Rm/\alpha\Zm$ is
$$%
  \Ical_{A} = \bigcup_{\theta \in A} \Ical_\theta.
$$
\end{defn}

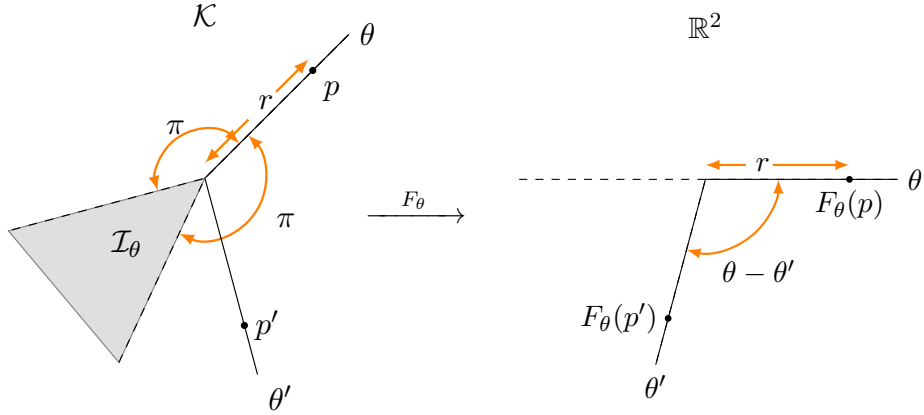
\begin{figure}[ht!]
\centering
\begin{minipage}{0.35\linewidth}

\begin{excise}{
  \begin{tikzpicture}[scale=2.5]
  \coordinate (O) at (0,0);
  \coordinate (A) at (45:1);
  \coordinate (P) at (45:.75);
  \coordinate (B) at (195:1);
  \coordinate (C) at (-115:1);
  \coordinate (D) at (-75:1);
  \coordinate (PP) at (-75:.75);

  \draw[arrows=latex-,thick,draw=orange] (0,.075)--(.3,.375) node[fill=white] {$r$};
  \draw[arrows=-latex,thick,draw=orange] (.35,.425)--( .51,.585);

  \filldraw[fill=lightgray,opacity=0.5] (O)--(B)--(C)--(O);
  \draw (A)--(O)--(B)[dashed] pic["$\pi$",draw=orange,solid,thick,latex-latex,
        angle eccentricity=1.2, angle radius=.6cm] {angle=A--O--B};
  \draw (C)--(O)--(A)[dashed] pic["$\pi$", right,draw=orange,solid,thick,
        latex-latex, angle eccentricity=1.2, angle radius=.75cm] {angle=C--O--A};
  \draw [solid] (O)--(D) node[below right]{$\theta'$};
  \draw [solid] (O)--(A) node[right] {$\theta$};
  \fill (P) circle[radius=.5pt] node[below right] {$p$};
  \fill (PP) circle[radius=.5pt] node[right] {$p'$};

  \node at (220:.5) {$\Ical_\theta$};
  \node at (90:.8) {$\KK$};

  \draw[opacity=0,black] (-1,-1.2) rectangle (1,1);
  \end{tikzpicture}
}\end{excise}
\includegraphics[width=1\textwidth]{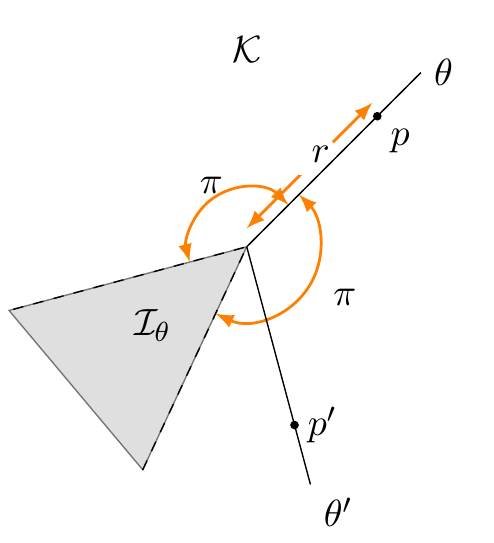}
\end{minipage}
\hspace{-8ex}
\begin{minipage}{0.15\linewidth}
  $$%
  \xrightarrow{\quad F_\theta \quad }
  $$
\end{minipage}
\begin{minipage}{0.36\linewidth}
\begin{excise}{
  \begin{tikzpicture}[scale=2.5]
  \coordinate (O) at (0,0);
  \coordinate (A) at (0:1);
  \coordinate (P) at (0:.75);
  \coordinate (B) at (180:1);
  \coordinate (D) at (-105:1);
  \coordinate (PP) at (-105:.75);

  \draw[arrows=latex-,draw=orange,thick] (0,.075)--(.3,.075) node[fill=white]{$r$};

  \draw[arrows=-latex,draw=orange,thick] (.35,0.075)--(.75,.075);

  \draw (A)--(O)--(D)[dashed] pic["$\theta-\theta'$",draw=orange,solid,thick,
        latex-latex,angle eccentricity=1.2, angle radius=.95cm,below] {angle=D--O--A};
  \draw [solid] (O)--(D) node[below]{$\theta'$};
  \draw [solid] (O)--(A) node[right] {$\theta$};
  \draw [dashed] (O) --(B);
  \fill (P) circle[radius=.5pt] node[below] {$F_\theta(p )$};
  \fill (PP) circle[radius=.5pt] node[left] {$F_\theta(p')$};

  \node at (90:.8) {$\Rm^2$};
  \draw[opacity=0,black] (-1,-1.2) rectangle (1,1);
\end{tikzpicture}
}\end{excise}
\includegraphics[width=1\textwidth]{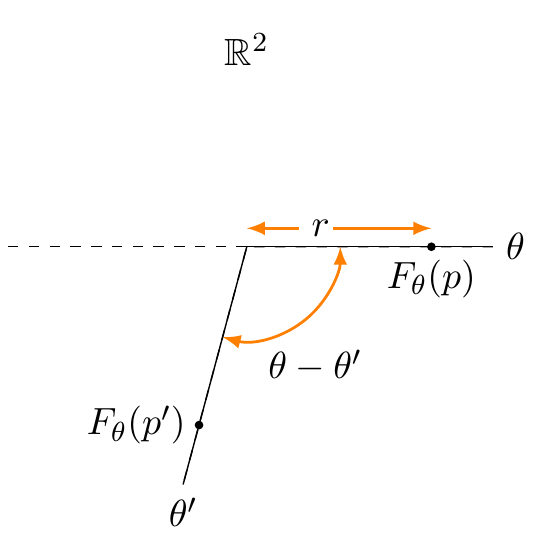}
\end{minipage}
\caption{\label{f:folding-map}%
Fix points $p\neq 0 \neq p'$ with angles $\theta$ and $\theta'$ on the
kale~$\KK$.  Left: The shadow $\Ical_\theta$ of~$p$ is the
interior of the sector of points whose shortest paths to~$p$ pass
through the origin.  In other words, as seen from $p$, the origin
casts the shadow $\Ical_\theta$.  All these points are
\emph{invisible} from $p$.  (For future reference, with notation as
in~(\ref{wpmdef}) and Lemma~\ref{lem:Dmtheta}, including the upper
dashed line gives~$\Ical^+_\theta$ and including the lower
dashed line gives~$\Ical^-_\theta$.)  Right: Under the folding
map $F_\theta$ centered at angle $\theta$, the shadow collapses to the
negative horizontal axis.}
\end{figure}

The terminology referring to (in)visibility and shadow is motivated as
follows.  Imagine placing a light source at a point $p = (r,\theta)$.
If rays of light (geodesics) in~$\KK$ are obstructed by the origin,
then $\Ical_{\theta}$ is the set of points in the shadow cast by the
origin.  Alternatively, imagine light emanating from sources within
$\Ical_{\theta}$: an observer at $(r,\theta)$ is not able to resolve
the image, since all light rays arriving at the observer have merged
at the origin.

\begin{rmk}
The folding map $F_\theta$ is the unique continuous map $\KK \to
\Rm^2$ that preserves all distances from points on the ray at
angle~$\theta$ to other points in~$\KK$; c.f.\
Lemma~\ref{l:distance-under-folding}.  In particular, it preserves
radius from the origin.  The folding map~$F_\theta$ collapses the part
of~$\KK$ invisible from~$\theta$ to the negative horizontal axis
of~$\RR^2$ and takes the fully visible part of~$\KK$ bijectively to
the complement of the negative horizontal axis.

The folding map~$F_\theta$ is the ``logarithm map'' from~$\KK$ to the
tangent space at any point with positive radius along the ray at
angle~$\theta$.  In smooth manifolds, log maps are right inverse to
exponential maps, the latter being globally defined on the tangent
space at a point~$p$, while the former is only defined in a
neighborhood of~$p$.  Here, singularity of the metric at~$\0 \in \KK$
prevents exp from being well defined, whereas uniqueness of geodesics
in~$\KK$ (that is, the absence of a cut locus) makes the log map
globally defined on~$\KK$.
\end{rmk}

\begin{excise}{%
  \begin{ex}
  If $|\theta - \hat \theta| \leq \pi$, then
  $$%
    \Ical_{[\theta,\hat\theta]} = \Ical_{[\hat \theta,\theta]}.
  $$
  That is, it matters not whether $\theta -\hat\theta < 0$ or~$\theta -
  \hat\theta>0$, since the subsets $[\theta,\hat\theta]$ and
  $[\hat\theta,\theta]$ of~$\KK$ coincide.
  \end{ex}
}\end{excise}%
%

\subsection{Barycenters and folded first moments}\label{sub:barycenters}

Let $\mu$ be a Borel probability measure on~$\KK$.  Our main results
concern statistics of random points drawn independently from the
measure~$\mu$ on~$\KK$.  We assume throughout that $\mu$ satisfies the
integrability condition
\begin{equation}\label{integ}
  \bar r := \int_\KK d(\0,p) \,d\mu(p) < \infty.
\end{equation}
Because $\KK$ is not a linear space, the mean of a probability
distribution on~$\KK$ cannot be defined using addition, as it can be
in $\Rm^2$.  Instead, we use the notion of barycenter of a
distribution~$\mu$.  If the second moment condition
(square-integrability)
\begin{equation}\label{squareint}
  \int_\KK d(\0,q)^2 \,d\mu(q) < \infty
\end{equation}
holds, then the function $\Gamma\colon\KK \to \Rm$ defined by
\begin{equation}\label{eq:gammaDef}
\Gamma(p) = \frac{1}{2} \int_\KK d(p,q)^2 \,d\mu(q)
\end{equation}
is finite for all $p \in \KK$, and it has a unique minimizer (proved
later, at Corollary~\ref{cor:Gammamax}).  This leads to the following
definition.

\begin{defn}\label{d:populationbarycenter}
Under the second moment condition (\ref{squareint}), the unique
minimizer of~$\Gamma$ is the \emph{barycenter} of~$\mu$, denoted
by~$\bar b$.
\end{defn}

It is possible to extend this definition in a consistent way to the
setting where only the integrability condition~(\ref{integ}) holds
for~$\mu$ rather than the stronger square-integrability
condition~\eqref{squareint}; see Definition~\ref{d:meanK}.  For now,
we only say enough to state this generalization of
Definition~\ref{d:populationbarycenter}, postponing the full
discussion to Section~\ref{s:moments}.

Under the folding map $F_\theta:\KK \to \Rm^2$, the measure $\mu$
pushes forward to a probability measure $\tilde \mu_\theta = \mu \circ
F_\theta^{-1}$ on $\Rm^2$.  The family of measures $\{\tilde
\mu_\theta\}_{\theta \in \Rm/\alpha\Zm}$ on $\Rm^2$ allows us to
deduce properties of the measure~$\mu$ on~$\KK$.  For points $z \in
\Rm^2$, we typically use cartesian coordinates $z = (z_1,z_2)$; the context should
prevent any confusion with the radial representation $(r,\theta)$ of
points in~$\KK$.  Back in~$\Rm^2$, denote by $e_1 = (1,0)$ and $e_2 =
(0,1)$ the standard basis vectors, and by ``$\cdot$'' the standard
inner product.  The mean of $\tilde\mu_\theta$ in~$\RR^2$ can be
defined in the usual way, as follows.

\begin{defn}\label{d:moments}
For $\theta \in \RR/\alpha\ZZ$, the \emph{first moment} of $\mu$
folded about $\theta$ (or equivalently, the \emph{mean} of $\mu$
folded about $\theta$) is
$$%
  m_\theta = \int_{\Rm^2} z \,d \tilde \mu_\theta(z) = \int_\KK
  F_\theta (p)\, d\mu(p)=  e_1 m_{\theta,1} + e_2  m_{\theta,2}
$$
where
$$%
  m_{\theta,i} = e_i  \cdot m_\theta = \int_\KK e_i \cdot F_\theta
  (p)\, d\mu(p)\quad\text{for }i = 1,2.
$$
\end{defn}

The integrability condition~\eqref{integ} implies that the first
moment $m_\theta$ is finite and that $\theta \mapsto m_{\theta}$ is
continuous.

\begin{defn}\label{d:stickyclassification}
Fix a probability distribution $\mu$ on~$\KK$ and let $K \subset
\RR/\alpha\ZZ$ be the subset on which $m_{\theta,1} \geq 0$.  The
distribution $\mu$ is
\begin{itemize}
\item[\it(i)]%
\emph{fully sticky} if~$K$ is empty;
\item[\it(ii)]%
\emph{partly sticky} if $K$ is non-empty and $m_{\theta,1} = 0$ on its
entirety;~and
\item[\it(iii)]%
\emph{nonsticky} if~$K$ has non-empty interior and $m_{\theta,1} > 0$ on $\text{int}(K)$.
\end{itemize}
The measure $\mu$ is \emph{sticky} if it is either fully sticky or
partly sticky.  When $\mu$ is partly sticky, a direction $\theta$ is
\emph{sticky} if $m_{\theta,1} < 0$ and \emph{fluctuating} if
$m_{\theta,1} \geq 0$.
\end{defn}
Notice that since  $\theta \mapsto m_{\theta,1}$ is continuous, the set $K$
from Definition~\ref{d:stickyclassification}  is
always a closed set.
To rule out certain pathologies, we always assume the following
nondegeneracy condition.

\begin{conv}\label{d:nondegen}
The measure $\mu$ is \emph{nondegenerate} in the sense that
\begin{equation}\label{nondegen}
  \mu(R_{\theta,\theta'}) < 1 \text{ for all angles } \theta,\theta'
  \text{ such that } |\theta - \theta'| \geq \pi,
\end{equation}
where for angles $\theta, \theta' \in \RR/\alpha\ZZ$,
$$%
  R_{\theta,\theta'} = \{ (r,\theta) \mid r \geq 0 \} \cup
  \{(r,\theta') \mid r \geq 0 \},
$$
the union of the two rays at angles $\theta$ and~$\theta'$.
\end{conv}

If nondegeneracy does not hold, then $\mu(R_{\theta,\theta'}) = 1$ for
some pair of angles $\theta, \theta' \in \Rm/\alpha\ZZ$ such that
$|\theta - \theta'| \geq \pi$: all of the mass is concentrated on two
rays separated by an angle of at least~$\pi$.  Since $|\theta -
\theta'| \geq \pi$ means that $(1,\theta') \in \ol{\Ical_\theta}$ (or
equivalently that $(1,\theta) \in \ol{\Ical_{\theta'}}$), it is not
hard to show that this scenario is metrically equivalent the case of
$\KK = \Rm$.

The terms fully sticky, partly sticky, and nonsticky in
Definition~\ref{d:stickyclassification} are mutually exclusive.  The
following result shows that under minimal assumptions, every
distribution is covered by one of these three cases; this is
essentially Proposition~\ref{p:cases}.

\begin{prop}\label{p:casesSimple}
If $\mu$ is a probability measure on $\KK$ that is
integrable~(\ref{integ}) and nondegenerate~(\ref{nondegen}), then
$\mu$ is either fully sticky, partly sticky, or nonsticky.
Furthermore, if $\mu$ is partly sticky, then the interval~$[A,B]$ on
which $m_{\theta,1} \geq 0$ has length $|A - B| < \pi$; if $\mu$ is
nonsticky, then $|A - B| \leq \pi$ and the function $\theta \mapsto
m_{\theta,1}$ is strictly concave on its interior~$(A,B)$.
\end{prop}

We are now in a position to generalize the concept of barycenter
in~$\KK$ to the setting where $\mu$ only satisfies the integrability
condition \eqref{integ} but not the square-integrability
condition~\eqref{squareint}.

\begin{defn}\label{d:meanK}
If the probability distribution $\mu$ satisfies \eqref{integ} and is
sticky (either fully or partly sticky), then set the \emph{mean} of
$\mu$ equal to the origin~$\0$.  If $\mu$ is nonsticky, then set the
\emph{mean} of $\mu$ equal to the point $(m_{\theta',1}\,,\theta') \in
\KK$, where $\theta'$ maximizes the function $\theta \mapsto
m_{\theta,1}$.
\end{defn}

In light of Proposition~\ref{p:casesSimple}, the mean of $\mu$ is
well defined for all distributions that satisfy the integrability and
nondegeneracy assumptions; the second moment condition used in the
definition of the barycenter is not necessary to define a mean.  In
Corollary~\ref{cor:Gammamax} we show that when the barycenter is
defined, the mean of~$\mu$ coincides with its barycenter.

\subsection{Empirical measures and the law of large numbers}\label{sub:empirical}

For a given set of points $\{p_n\}_{n=1}^N \subset \KK$, define the
empirical measure
$$%
  \mu_N = \frac{1}{N} \sum_{n=1}^N \delta_{p_n},
$$
the averaged sum of unit measures supported on the points~$p_n$.  This
is a Borel probability measure on $\KK$, and all results of the
previous section apply to~$\mu_N$.  Let $b_N = b(p_1,\ldots,p_N)$ be
the barycenter of~$\mu_N$:
\begin{equation}\label{d:sample-barycenter-kale}
  b_N = b(p_1,\dots,p_N) = \argmin_{p \in \KK} \Bigl( \frac{1}{2N}
  \sum_{n=1}^N d( p, p_n)^2\Bigr),
\end{equation}
uniquely defined (by Corollary~\ref{cor:Gammamax}).  For $\theta \in
\RR/\alpha\ZZ$, write $\eta_{\theta,N} \in \Rm^2$ for the
\mbox{\emph{folded~average}}
\begin{equation}\label{etaNdef}
  \eta_{\theta,N} = \frac{1}{N} \sum_{n=1}^N F_\theta(p_n).
\end{equation}
The folded first moments of~$\mu_N$, which we denote by~$m_\theta^N
\in \Rm^2$, are defined by
$$%
  m^N_\theta = e_1 \,  m^N_{\theta,1}  + e_2\, m^N_{\theta,2}
$$
where
$$%
  m^N_{\theta,i} = e_i \cdot m^N_\theta = \int_\KK e_i \cdot F_\theta
  (p)\, \,d\mu_N(p) \quad\text{for $i=1,2$}.
$$
Comparing these formulas to (\ref{etaNdef}), the folded average is
evidently equivalent to the first moment of the empirical measure:
\begin{equation}\label{etaNMequiv}
\eta_{\theta,N} = m^N_\theta \quad\text{ for all }\theta \in \Rm/\alpha\Zm.
\end{equation}
An important issue in our analysis is whether the folded average
$\eta_{\theta,N}$ is close to the folded barycenter $F_{\theta} b_N$,
that is, whether ``averaging commutes with folding".  These two points
in $\Rm^2$ may not coincide; the relation between $\eta_{\theta,N}$
and $F_{\theta} b_N$ is addressed later in Lemma~\ref{lem:bNetaN2}.

Henceforth, let $\{p_n\}_{n=1}^N$ be a collection of independent
random points on $\KK$, each distributed according to $\mu$.  More
precisely, let $\{p_n(\omega) \mid n = 1,\ldots,N\}$ be a collection
of independent, identically distributed $\KK$-valued random variables,
each distributed according to~$\mu$ over a probability space
$(\Omega,{\mathcal A},\Pm)$.  Their barycenter $b_N(\omega) =
b(p_1(\omega),\dots,p_N(\omega)) \in \KK$ is also a random variable
taking values in $\KK$.  For each $\theta\in\RR/\alpha\ZZ$, let
$m^N_\theta=m^N_\theta(\omega)$ be the random first moments associated
with the empirical measures $\mu_N = \mu_N(\omega) =
\frac{1}{N}\sum_{n=1}^N\delta_{p_n(\omega)}$.  As before, denote by
$m_\theta$ the deterministic folded means of~$\mu$ in
Definition~\ref{d:moments}.  For any angle $\theta$,
$$%
  \expE[ m_{\theta}^N] = \frac{1}{N} \sum_{n=1}^N \int_{\KK}
  F_{\theta}(p_n) \,d\mu(p_n) = \frac{1}{N} \sum_{n=1}^N \int_{\Rm^2}
  z \,d \tilde \mu_{\theta}(z) = m_{\theta}.
$$
By the usual strong law of large numbers for $\Rm^2$-valued random
variables,
\begin{equation}\label{moments:slln}
  m_{\theta}^N \rightarrow m_\theta\ \Pm\text{-almost surely as }
  N\to \infty, \text{for all }\theta\in \RR/\alpha \ZZ.
\end{equation}

Translating back into a law of large numbers in $\KK$ for the random
barycenters~$b_N$, the behavior in the first two cases is strikingly
different than the typical law of large numbers in a Euclidean space.
The following result is proved in Section~\ref{s:slln}.

\begin{thm}[Law of Large Numbers on $\KK$]\label{t:LLN-intro}
Assume that $\mu$ satisfies the integrability condition~(\ref{integ}).
Exactly one of the following holds, depending on how sticky $\mu$~is.
\begin{enumerate}
\item%
(Fully sticky) %
The mean of $\mu$ is $\0$ and there exists a random integer $N^*$ such
that the barycenter $b_N$ from~(\ref{d:sample-barycenter-kale})
satisfies $b_N(\omega) = \0$ for all $N \geq N^*(\omega)$,
$\Pm$-almost surely.
\item%
(Partly sticky) %
The mean of $\mu$ is $\0$ and $b_N(\omega) \rightarrow \0$ almost
surely as $N \rightarrow \infty$.  Furthermore, if $[A,B]$ is the
interval of fluctuating directions from
Definition~\ref{d:stickyclassification}(ii) and Proposition \ref{p:casesSimple}, and $I$ is an open
interval of angles containing $[A,B]$ then there exists a random
integer $N^*$ such that $b_N(\omega) \in C_I$
(Definition~\ref{d:sector}) for all $N \geq N^*(\omega)$, almost
surely.
\item%
(Nonsticky) %
The mean $\bar b$ of $\mu$ is not $\0$ and $b_N(\omega) \rightarrow
\bar b$ almost surely as $N \rightarrow \infty$.
\end{enumerate}
\end{thm}

The theorem implies that for all of the sticky directions~$\theta$,
the empirical mean $b_N$ stops fluctuating after some random but
finite $N^*$ along the ray $\{ (r,\theta)\;|\; r \geq 0 \}$; this is
the phenomenon that we refer to as ``stickiness".  In fluctuating
directions, the empirical mean $b_N$ continues to vary as $N \to
\infty$, although the magnitude of the movement goes to
zero~asymptotically.

\subsection{Central Limit Theorems}\label{sec:central-limit-theorem}

The central limit theorems in this section describe the asymptotic
behavior of the properly normalized fluctuations of~$b_N$ about the
mean of~$\mu$.  Due to the non-standard nature of the sticky law of
large numbers, it is not surprising that the central limit theorem
also takes a different form in sticky cases.  Even in the nonsticky
case, the central limit theorem is non-standard.  Each of the three
possibilities in Proposition~\ref{p:casesSimple} is covered in a
separate theorem; these three theorems are proved in
Section~\ref{s:clt1}.

\subsubsection{Fully sticky case}\label{sub:fully}

The simplest case is the fully sticky case, where there are
asymptotically no fluctuations in any direction.   On $\KK$
define the scaling $\beta (r,\theta) = (\beta r,\theta)$ for arbitrary
$\beta \geq 0$ such that $F_{\theta'}\big(\beta (r,\theta)\big) =
\beta F_{\theta'} (r,\theta)$ for all $\theta,\theta' \in
\Rm/\alpha\Zm$ and $r,\beta \geq 0$.  Let $\nu_N$ denote the
distribution of the rescaled empirical means:
\begin{equation}\label{eq:rescaled-emp-means}
  \nu_N(U) = \Pm\big(\sqrt{N} b_N(\omega) \in U \big),
  \text{ where $b_N$ is the empirical barycenter~(\ref{d:sample-barycenter-kale})}
\end{equation}
for all Borel sets $U \subset \KK$.

\begin{thm}\label{t:fullstickyclt}
If a probability measure $\mu$ on~$\KK$ is fully sticky, then the
rescaled empirical mean measures $\{\nu_N\}_{N=1}^\infty$
from~(\ref{eq:rescaled-emp-means}) converge in the total variation
norm (and hence weakly) to the point measure $\delta_{\0}$ as $N \to
\infty$.  In particular, for any bounded function $\phi:\KK \to \Rm$,
\begin{equation}\label{stickyclt0}
  \lim_{N \to \infty} \int_\KK \phi(p) d\nu_N(p)
  =
  \phi(\0)\,.
\end{equation}
\end{thm}

In this fully sticky case, the term ``Central Limit Theorem'' is a bit
of a misnomer, since there are no asymptotic fluctuations.  In fact,
Theorem \ref{t:fullstickyclt} would still be true if we replace
$\sqrt{N}$ in (\ref{eq:rescaled-emp-means}) with any increasing
function of $N$.

The next two cases require a bit more notation and setup.

\subsubsection{Partly sticky case}\label{sub:partly}

Assume the second moment condition~(\ref{squareint}).  Since the
mean~$\bar b$ of~$\mu$ lies at the origin~$\0$ in the partly sticky
case, again consider the rescaled empirical measure $\nu_N$ defined
by~(\ref{eq:rescaled-emp-means}).  The limit of~$\nu_N$ is another
measure on~$\KK$, constructed as follows.

Let $\theta^*$ and $\rho \in [0,\pi/2)$ be such that $[\theta^*-\rho,
\theta^*+\rho]=[A,B]$ where $[A,B]$ is the interval of fluctuating
directions (Definition~\ref{d:stickyclassification}.ii and Proposition \ref{p:casesSimple}).  Let $g$
denote the law of the multivariate normal random variable on $\Rm^2$
having mean zero and covariance matrix
\begin{equation}\label{eq:g}
  \Sigma = \int_{\Rm^2} y y^T d \tilde \mu_{\theta^*}(y).
\end{equation}
This matrix is well defined due to the square-integrability
condition~(\ref{squareint}).

Denote by $D_\rho \subset \Rm^2$ the closed sector
\begin{equation}\label{eq:Drho}
  D_\rho = \left\{(r \cos\vartheta, r \sin\vartheta) \in \Rm^2 \mid r
  \geq 0 \text{ and $-\rho$} \leq \vartheta \leq \rho\right\}
\end{equation}
and by $\hat P_\rho: \RR^2 \to D_\rho$ the \emph{convex projection}
onto~$D_\rho$:
\begin{equation}\label{eq:convex projection}
  \hat P_\rho(q) = \argmin_{z \in D_{\rho}} d_2(q,z),
\end{equation}
where $d_2(z,w): \Rm^2 \times \Rm^2 \to [0,\infty)$ denotes the
Euclidean metric in~$\Rm^2$.  Since $|A-B| < \pi$, the folding map
$F_{\theta^*}$ takes the sector $C_{[A,B]}$
(Definition~\ref{d:sector}) bijectively to $D_\rho$.  It is possible
that $\rho =0$ or equivalently $A=B$, in which case $C_{[A,B]}$ and
$D_{\rho}$ are rays.

Finally, define the measure $h_{\theta^*}$ on $\KK$ by
\begin{equation}\label{eq:h}
  h_{\theta^*} = g \circ \hat P_\rho^{-1} \circ F_{\theta^*},
\end{equation}
where $g \circ \hat P_\rho^{-1}$ is the pushforward of the normal
measure~$g$, whose covariance matrix is defined in~(\ref{eq:g}), under
the projection $\hat P_\rho$ to~$D_\rho$.
Figure~\ref{f:partly-sticky} illustrates the construction in an
example.

\begin{figure}[ht!]
\centering
\begin{minipage}{0.28\linewidth}
\begin{excise}{
  \begin{tikzpicture}[scale=2.5]
  \coordinate (O) at (0,0);
  \coordinate (A) at (80:1);
  \coordinate (Ae) at (95:1);
  \coordinate (theta) at (25:1);
  \coordinate (B) at (-30:1);
  \coordinate (Be) at (-45:1);
  \coordinate (D) at (-50:1);

  \filldraw[fill=lightgray,opacity=0.5] (O)--(B) arc (-30:80:1) (A)--(O);
  \draw[solid](A)node[right]{\raisebox{2.5ex}{$\!\!\!B$}}--(O)--(B)node[right]{$\!A$};
  \draw[dotted] (theta) node[right]{$\theta^*$}--(O);

  \draw [dashed] (Ae) node[left]{$B^\epsilon\!$} --(O) --(Be) node[left]{$A_\epsilon$};

  \node at (-0:.60) {$C_{[A,B]}$};
  \node at (85:1.3) {$\KK$};
  \draw[opacity=0,black] (-.5,-1.25) rectangle (1.1,1.35);
  \end{tikzpicture}
}\end{excise}
\includegraphics[width=1\textwidth]{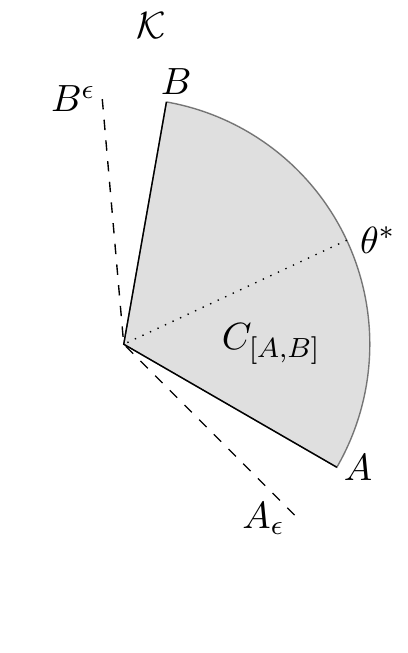}
\end{minipage}
\begin{minipage}{0.2\linewidth}
$$%
  \xrightarrow{\quad F_{\theta^*} \quad}
$$
\end{minipage}
\begin{minipage}{0.35\linewidth}
\hspace{-2ex}
\begin{excise}{
  \begin{tikzpicture}[scale=2.5]

  \coordinate (one) at (1,0) ;
  \coordinate (O) at (0,0);
  \coordinate (A) at (55:1) ;
  \coordinate (B) at (-55:1) ;
  \coordinate (Ae) at (70:1) ;
  \coordinate (Be) at (-70:1) ;

  \coordinate (A1) at ($(55:.1)!.6cm!90:(55:1)$);
  \coordinate (A2) at ($(55:.3)!.6cm!90:(55:1)$);
  \coordinate (A3) at ($(55:.5)!.6cm!90:(55:1)$);
  \coordinate (A4) at ($(55:.7)!.6cm!90:(55:1)$);
  \coordinate (A5) at ($(55:.9)!.6cm!90:(55:2)$);

  \coordinate (B1) at ($(-55:.1)!.6cm!-90:(-55:1)$);
  \coordinate (B2) at ($(-55:.3)!.6cm!-90:(-55:1)$);
  \coordinate (B3) at ($(-55:.5)!.6cm!-90:(-55:1)$);
  \coordinate (B4) at ($(-55:.7)!.6cm!-90:(-55:1)$);
  \coordinate (B5) at ($(-55:.9)!.6cm!-90:(-55:1)$);

  \coordinate (C1) at (155:.6cm) ;
  \coordinate (C2) at (180:.6cm) ;
  \coordinate (C3) at (205:.6cm) ;

  \filldraw[fill=lightgray,opacity=0.5] (O)--(B) arc (-55:55:1) (A) --(O);

  \draw [dotted] (0:1)--(O);
  \draw [solid]  (A)--(O)--(B)  ;
  \draw [dashed]  (Ae)--(O)--(Be)  ; 

  \draw[orange,dashed,arrows=-latex,thick,shorten >= 2mm] (A1)--($(O)!(A1)!(A)$);
  \draw[orange,dashed,arrows=-latex,thick,shorten >= 2mm] (A2)--($(O)!(A2)!(A)$);
  \draw[orange,dashed,arrows=-latex,thick,shorten >= 2mm] (A3)--($(O)!(A3)!(A)$);
  \draw[orange,dashed,arrows=-latex,thick,shorten >= 2mm] (A4)--($(O)!(A4)!(A)$);
  \draw[orange,dashed,arrows=-latex,thick,shorten >= 2mm] (A5)--($(O)!(A5)!(A)$);

  \draw[orange,dashed,arrows=-latex,thick,shorten >= 2mm] (B1)--($(O)!(B1)!(B)$);
  \draw[orange,dashed,arrows=-latex,thick,shorten >= 2mm] (B2)--($(O)!(B2)!(B)$);
  \draw[orange,dashed,arrows=-latex,thick,shorten >= 2mm] (B3)--($(O)!(B3)!(B)$);
  \draw[orange,dashed,arrows=-latex,thick,shorten >= 2mm] (B4)--($(O)!(B4)!(B)$);
  \draw[orange,dashed,arrows=-latex,thick,shorten >= 2mm] (B5)--($(O)!(B5)!(B)$);

  \draw[orange,dashed,arrows=-latex,thick,shorten >= 2mm] (C1)--(O);
  \draw[orange,dashed,arrows=-latex,thick,shorten >= 2mm] (C2)--(O);
  \draw[orange,dashed,arrows=-latex,thick,shorten >= 2mm] (C3)--(O);

  \draw pic[draw=orange,arrows={latex[orange,fill=orange]-latex[orange,fill=orange]},
	solid,thick, angle eccentricity=1.05, angle radius=2cm,right,
	"$\,\rho=\frac{|A-B|}2$"] {angle=one--O--A};

  \node at (-20:.60) {$\mathcal{D}_\rho$};
  \node at (85:1.3) {$\Rm^2$};
  \draw[opacity=0,black] (-.5,-1.25) rectangle (1,1.35);
  \end{tikzpicture}
}\end{excise}
\includegraphics[width=1\textwidth]{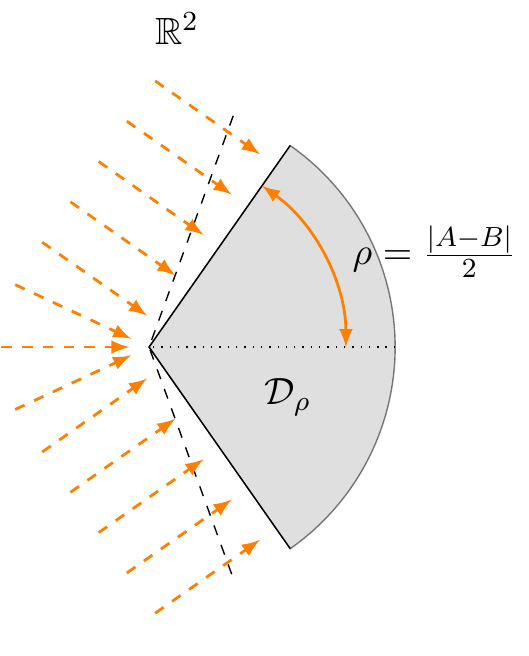}
\end{minipage}
\caption{\label{f:partly-sticky}%
Partly sticky case.  Left: $C_{[A,B]}$ is that sector in $\KK$
centered at $\theta^*$ that is spanned by the angles $\theta$ for
which $m_{\theta,1}=0$.  For $N$ larger than a finite but random
number, $b_N \in C_{[A_\epsilon,B^\epsilon]}$ almost surely.  Right:
$D_\rho$ is the bijective image of $C_{[A,B]}$ under the folding map
centered at $\theta^*$.  With a Gaussian $g$ centered at $0\in \RR^2$,
up to the bijection, the limiting measure is $g$ on ${\rm
int}(D_\rho)$ and the pushforward of $g$ on $\RR^2 \setminus D_\rho$
to $\partial D_\rho$ under the convex projection $\hat P_\rho$.  The
dashed arrows show the directions of this convex projection.}
\end{figure}
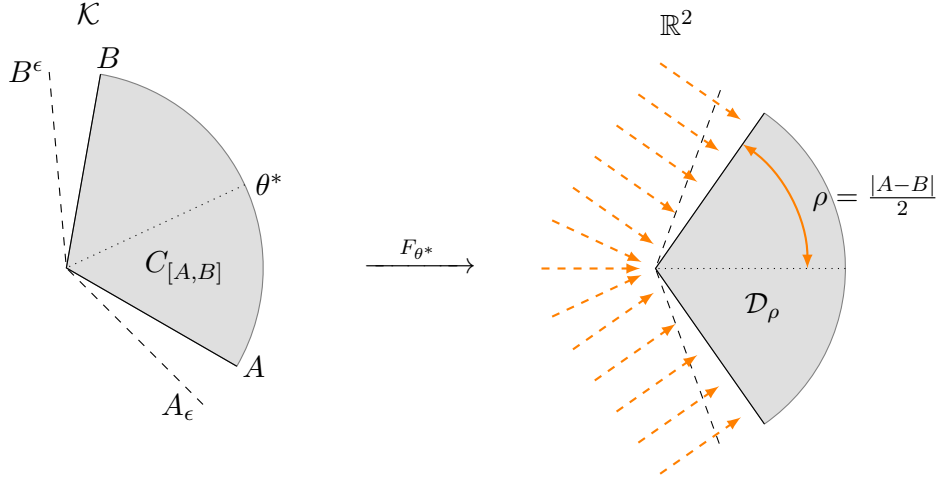

\begin{thm}\label{t:stickyclt}
If a measure $\mu$ on~$\KK$ is partly sticky and
square-integrable~\eqref{squareint}, then the rescaled empirical mean
measures $\{\nu_N\}_{N=1}^\infty$ from~(\ref{eq:rescaled-emp-means})
converge weakly to the measure $h_{\theta^*}$ from~(\ref{eq:h}) as $N
\to \infty$, where $\theta^*$ is the midpoint of the interval~$K$ in
Definition~\ref{d:stickyclassification}.  That is, for any continuous,
bounded function $\phi:\KK \to \Rm$,
\begin{equation}\label{stickyclt1}
  \lim_{N \to \infty} \int_\KK \phi(p) d\nu_N(p)
  =
  \int_\KK \phi(p) dh_{\theta^*}(p).
\end{equation}
\end{thm}

The measure $h_{\theta^*}$ is supported on the closed sector
$C_{[A,B]}$.  The limit distribution $h_{\theta^*}$ can be decomposed
into a singular part and an absolutely continuous part:
$$%
  h_{\theta^*} = h_\mathit{sing} + h_\mathit{abs}.
$$
The absolutely continuous part is the restriction of a Gaussian to the
set $\text{int}(C_{[A,B]})=C^+_{(A,B)}$, which is the interior of the
closed sector $C_{[A,B]}$:
$$%
  h_\mathit{abs}(V) = g \circ F_{\theta^*}\big(V \cap C^+_{(A,B)}\big).
$$
When $A = B$, the sector $C_{[A,B]}$ has no interior and
$h_\mathit{abs} = 0$.  The singular part $h_\mathit{sing}$ is
supported on the boundary $\partial C_{[A,B]}$, and it includes an
atom $w \delta_0(p)$ at the origin with weight
$$%
  w = g\left(\left\{(r\cos\vartheta,r\sin\vartheta) \in \Rm^2 \mid r
  > 0 \text{ and }\vartheta \in [\rho +\pi/2, 3\pi/2-\rho]\right\}\right).
$$
However, not all of the mass in the singular part lies at the origin;
$h_\mathit{sing}$ also distributes mass continuously on the edges of
the sector $C_{[A,B]}$.  In particular,
$$%
  h_\mathit{sing}\big(\partial C_{[A,B]} \setminus \{0\}\big)
  =
  g\!\left(\left\{\!
  (r\cos\vartheta,r\sin\vartheta) \in \Rm^2 \!\mid\! r > 0, \vartheta
  \in [\rho, \rho +\pi/2)\cup(3\pi/2-\rho,2\pi-\rho]\right\}\!
  \right)\!.
$$

\subsubsection{Nonsticky case}\label{sub:non-sticky}

When $\mu$ is nonsticky, the mean of $\mu$ is $\bar b =
(r^*,\theta^*) \in \KK$, where $r^* = m_{\theta^*,1} > 0$ and
$\theta^*$ is the unique angle for which
$$%
  m_{\theta^*,1} = \max_\theta m_{\theta,1}.
$$
In particular this means that $\bar b \neq \0$, so the limit measure
obtained by renormalizing fluctuations of~$b_N$ lives on the tangent
space of~$\bar b$, which is isomorphic to~$\Rm^2$, not $\KK$ as in
sticky~cases.

With $\theta^*$ fixed, the family of random variables
$\{m_{\theta^*}^N\}_{N=1}^\infty$ satisfies a standard central limit
theorem in $\Rm^2$.  Specifically, let $g$ be the law of a
multivariate normal random variable on $\Rm^2$ with zero mean and
covariance matrix
$$%
  \Sigma = \int_{\Rm^2} (y- F_{\theta^*} \bar b) (y-F_{\theta^*} \bar
  b)^T d \tilde \mu_{\theta^*}(y).
$$
This matrix is well defined under the square-integrability
condition~(\ref{squareint}).  The standard central limit theorem
implies that as $N \to \infty$ the law of the random variable
$$%
  \sqrt{N} \left(m_{\theta^*}^N- F_{\theta^*} \bar b\right)
$$
in $\Rm^2$ converges weakly to $g$.

Although is it reasonable to expect that $F_{\theta^*} b_N$ would
satisfy the same central limit theorem, this might in fact not be the
case, depending on whether the closed shadow $\ol{\Ical_{\theta^*}}$
carries mass.  Define $\kappa \geq 0$ to be the random variable
\begin{equation}\label{eq:kappa}
\kappa(\omega)
  = \left \{
  \begin{array}{cc}
  \dfrac{w^+(\theta^*)}{r^*}&\text{if } e_2 \cdot F_{\theta^*}
  b_N(\omega) < 0,\\
\\
  \dfrac{w^-(\theta^*)}{r^*}&\text{if } e_2 \cdot F_{\theta^*} b_N(\omega) > 0,\\ \\
\;\; 0 & \quad \text{else},
  \end{array} \right.
\end{equation}
where (cf.~Figure~\ref{f:folding-map} for $\Ical^\pm_\theta$)
\begin{align}
\nonumber
  w^\pm(\theta) &= \int_{\Ical_\theta^\pm} d(\0,p) \,d\mu(p),
\end{align}
and
\begin{equation}\label{wpmdef}
  \begin{aligned}
\Ical_\theta^+ &= \KK \setminus \{(r,\theta') \mid r > 0 \text{
and } -\pi \leq \theta' - \theta <\pi \},
\\
\Ical_\theta^- &= \KK \setminus \{(r,\theta') \mid r > 0 \text{ and }
-\pi < \theta' - \theta \leq \pi \}.     
  \end{aligned}
\end{equation}
On the Borel sets $W$ in~$\Rm^2$ define the family
$$%
  \tilde \nu_N(W) = \Pm\big(\sqrt{N} (e_1 \cdot F_{\theta^*} b_N -
  r^*, (1 + \kappa)e_2 \cdot F_{\theta^*} b_N) \in W\big)
$$
of measures indexed by~$N$.  If
$\mu(\ol{\Ical_{\theta^*}}) = 0$, then $\kappa = 0$ and
$$%
 \tilde \nu_N(W) = \Pm\big(\sqrt{N} (F_{\theta^*} b_N -
 F_{\theta^*} \bar b)\in W \big),
$$
since $F_{\theta^*} \bar b = (r^*,0)$.

\begin{thm}\label{t:non-stickyclt}
If $\mu$ is nonsticky and square-integrable \eqref{squareint}, then
the measures $\{\tilde \nu_N\}_{N=1}^\infty$ converge weakly to $g$ as
$N \to \infty$.  That is, for any continuous, bounded function
$\phi:\Rm^2 \to \Rm$,
\begin{equation}\label{stickyclt2}
  \lim_{N \to \infty} \int_\Rm \phi(z) d\tilde \nu_N(z)
  =
  \int_\Rm \phi(z) dg(z).
\end{equation}
\end{thm}

When $w^+(\theta^*) = w^-(\theta^*)$, Theorem~\ref{t:non-stickyclt}
implies that $\sqrt{N} (F_{\theta^*} b_N - F_{\theta^*} \bar b)$ is
Gaussian in the limit as $N \to \infty$.  When $w^+(\theta^*) =
w^-(\theta^*) > 0$, the fluctuation of $F_{\theta^*} b_N -
F_{\theta^*} \bar b$ in the $e_2$ direction is smaller than the
fluctuation of $m_{\theta^*,2}^N$; this is due to the presence of mass
in the closed shadow $\ol{\Ical_{\theta^*}}$.  On the other hand, if
$w^+(\theta^*) \neq w^-(\theta^*)$, then $\sqrt{N} (F_{\theta^*} b_N -
F_{\theta^*} \bar b)$ is not Gaussian in the limit; see
Example~\ref{e:Gaussian}.

\section{Examples}\label{s:examples}

Here are a few examples illustrating some phenomena described by the
limit theorems.

\begin{ex}[Partly sticky]\label{ex:sector}
Fix $\alpha > 2\pi$ and $\theta^*\in\Rm/\alpha\Zm$.  Let $K \geq 3$ be
an odd integer.  Let $\mu$ be the sum of $K$ atoms having mass $1/K$
at the points
$$%
  q_k
  =
  (1,\theta^*+\frac{2\pi}{K}k), \quad k
  =
  -\frac{K-1}{2}, \ldots, 0, \ldots, \frac{K-1}{2} \in \mathbb{Z}.
$$
That is,
$$%
  \mu = \frac{1}{K}  \sum_{k=1}^K \delta_{q_k}.
$$
In this case $m_\theta \leq 0$ for all $\theta \in\Rm/\alpha\Zm$,
while $m_\theta = 0$ if and only if $|\theta-\theta^*| \leq \pi/K$.
The limit distribution $h_{\theta^*}$ is supported on the sector
$$%
  C_{[-\frac{\pi}{K}, \frac{\pi}{K}]} = \Big\{(r,\theta)\:\Big|\: r \geq 0
  \text{ and }-\frac\pi K \leq \theta-\theta^* \leq \frac\pi K\Big\},
$$
including a singular part at the origin with weight $1-\frac{2}{K}\lfloor \frac{K+2}{4}\rfloor-\frac{1}{K}$
and a singular part on $\partial C_{[-\frac{\pi}{K}, \frac{\pi}{K}]} \setminus \{0\}$, with weight $\frac{2}{K}\lfloor \frac{K+2}{4}\rfloor$
cf.~Figure~\ref{f:ex1}.  The limit distribution does not vary
with~$\alpha$, given that $\alpha > 2\pi$.

\end{ex}

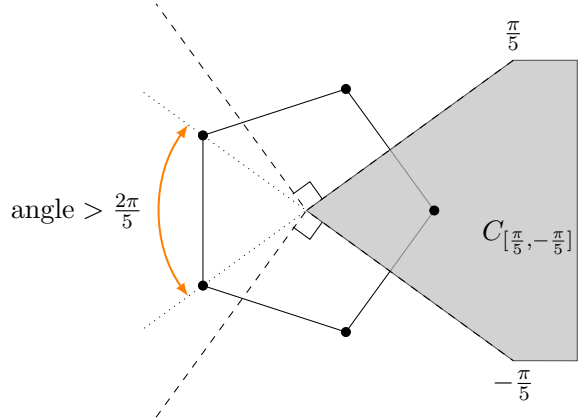
\begin{figure}[ht]
\centering
\begin{excise}{
\begin{tikzpicture}[scale=1.75]
  \coordinate (O) at (0,0);
  \coordinate (A) at (0:1);
  \coordinate (B) at ({360/5}:1);
  \coordinate (C) at ({2*360/5}:1);
  \coordinate (D) at ({-(2*360/5)}:1);
  \coordinate (E) at (-{360/5}:1);

  \coordinate (a) at ({180/5}:2.0);
  \coordinate (b) at (-{180/5}:2.0);
  \coordinate (aperp) at ({90+180/5}:2.0);
  \coordinate (bperp) at ({-90-180/5}:2.0);
  \coordinate (aa) at ($(a) +(.5,0)$);
  \coordinate (bb) at ($(b) +(.5,0)$);

  \draw[solid] (A)--(B) --(C)-- (D)--(E)--(A);

  \filldraw[fill=lightgray,opacity=0.7]
           (a)--(O)--(b) --(bb)--(aa)-- (a);
  \draw[dashed] (a) node[above]{$\frac{\pi}5$}-- (O) --(aperp);
  \draw[dashed] (bperp)--(O)--(b) node[below]{$-\frac{\pi}5$};

  \draw[right angle length=1ex,right angle symbol={a}{O}{aperp}];
  \draw[right angle length=1ex,right angle symbol={b}{O}{bperp}];

  \node at (-7:1.75) {$C_{[\frac\pi5, -\frac\pi5]}$};

  \draw[dotted,shorten >=-1cm,shorten <=-1cm] (C)--(O)--(D);
  \draw pic[draw=orange,arrows={latex[orange,fill=orange]-latex[orange,
        fill=orange]},solid,thick, angle eccentricity=1.05, angle
        radius=2cm,left,"$\text{angle}>\frac{2\pi}{5}$"]{angle=C--O--D};

  \filldraw (A) circle (1pt);
  \filldraw (B) circle (1pt);
  \filldraw (C) circle (1pt);
  \filldraw (D) circle (1pt);
  \filldraw (E) circle (1pt);
\end{tikzpicture}
}\end{excise}
\includegraphics[width=0.5\textwidth]{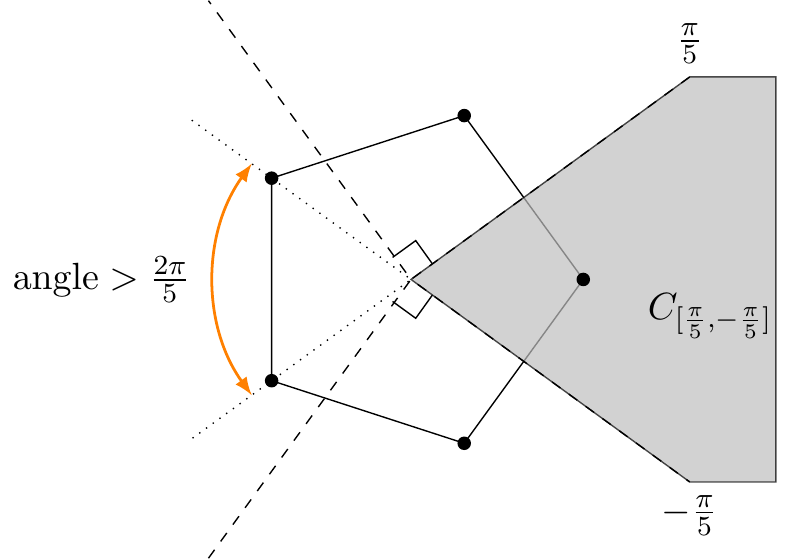}
\caption{\label{f:ex1}%
Example~\ref{ex:sector} in the case $K=5$, with $\theta^* = 0$.}
\end{figure}

In the limit, Example~\ref{ex:sector} gives the following.

\begin{ex}[Partly sticky with singular limit measure]
Fix $\alpha > 2\pi$ and $\theta^*\in\Rm/\alpha\Zm$.  Suppose $\mu$ is
uniform on the set
$$%
  S_1 = \{(r,\theta)\in\KK \mid r = 1\text{ and }-\pi < \theta-\theta^* < \pi\}.
$$
Then $m_\theta \leq 0$ for all $\theta \in \Rm/\alpha\Zm$, while
$m_\theta = 0$ only for $\theta = \theta^*$.  The limit distribution
$h_{\theta^*}$ puts an atom of mass $1/2$ at the origin, and half a
Gaussian on the ray $\{(r,\theta^*) \mid r > 0 \}$.  In particular,
$h_{\theta^*}$ has no absolutely continuous part.  As in
Example~\ref{ex:sector}, the limit distribution does not vary
with~$\alpha$, given that $\alpha > 2\pi$.
\end{ex}

\begin{ex}[Embedding the spider]
Suppose $\alpha > K\pi$.  Then there are angles $\theta_k \in
\Rm/\alpha\Zm$ for $k=1,\ldots,K$ such that $|\theta_k -\theta_j| >
\pi$ for all $j \neq k$.  Working with measures supported on the union
of the rays at angles $\theta_1,\ldots,\theta_K$ is equivalent to
working with probability distributions on the spider with $K$
legs---that is, an open book of dimension~$1$ with $K$ leaves,
cf.~\cite{HHMMN13}---by mapping the ray $\{(r,\theta_k) \in \KK \mid r
> 0\}$ to a leg of the spider.
\end{ex}

\begin{ex}[Full stickiness without square-integrability]\label{ex:no-mean}
Let $d\sigma= r dr \otimes d\theta$ denote the canonical measure on
$\KK$.  Here, $dr$ denotes the usual Lebesgue measure on $[0,\infty)$
and $d\theta$ the canonical quotient measure on $\Rm/\alpha\Zm$.  With
arbitrary but fixed $1 < \beta < 2$ let $\mu$ be the measure on $\KK$ with
density
$$%
  g(r,\theta)
  =
  \frac{2 \beta}{\alpha(\beta + 2)} \times
	\begin{cases}
	   \hfill 1 \hfill   &\text{if }0 \leq r \leq 1\\
	\frac{1}{r^{\beta+2}}&\text{if }1 \leq r < \infty.
	\end{cases}
$$
The integrability condition (\ref{integ}) is satisfied with $\bar r = \frac{2 \beta(\beta + 2)}{3(\beta + 1)(\beta - 1)}$.  Moreover, $m_{\theta,1} =  (2\pi -
\alpha)\bar r < 0$ for all $\theta\in \Rm/\alpha\Zm$.  By virtue of
Theorem~\ref{t:fullstickyclt}, there is a random integer $N^*$ such
that $b_N = \0$ for all $N\geq N^*$ almost surely.  On the other hand,
square-integrability does not hold, as $\int_{\KK} r_p^2\,d\mu(p) =
\infty$, and hence $\bar b$ is not defined.
\end{ex}

\begin{ex}[Non-Gaussian behavior in the nonsticky case]\label{e:Gaussian}
Fix $t > 3$ and let $\mu$ be the distribution on $\KK$ which puts mass $1/5$ at each of the points
\begin{align*}
  p_1 = (t,0),\quad p_2 = (1,\pi/2),\quad p_3 = (1,\pi), \quad p_4 = (2,- \pi), \quad  p_5 = (1,- \pi/2).
\end{align*}
The points $p_3$ and $p_4$ lie on the boundary of $\Ical_{\theta = 0}$, so under the folding map $F_{\theta}$ with $\theta = 0$, the points $p_3$ and $p_4$ collapse onto the axis $(-\infty,0) \times \{0\}$; points $p_2$ and $p_5$ map to the vertical axis $\{0\} \times \Rm$. We compute:
\[
m_{\theta = 0,1} = \frac{1}{5}(t + 0 - 1 - 2 + 0) = \frac{t - 3}{5} > 0.
\]
The push-forward $\tilde \mu_\theta = \mu \circ F_{\theta}^{-1}$ has symmetry about the $x$-axis when $\theta = 0$, which implies that $m_{\theta = 0,2} = 0$. By the results of Section \ref{s:moments} below, this implies $\theta \mapsto m_{\theta,1}$ is maximized at $\theta = \theta^* = 0$. However,
\[
w^+(\theta^*) = \frac{1}{5} d(\0,p_3)^2 = \frac{1}{5} \quad \text{and} \quad w^-(\theta^*) = \frac{1}{5} d(\0,p_4)^2 = \frac{4}{5}
\] 
in this case. As a consequence of Theorem~\ref{t:non-stickyclt} and the subsequent
remarks, the limit distribution of $\sqrt{N}(F_{\theta^*} b_N -
F_{\theta^*} \bar b)$ on $\Rm^2$ is non-degenerate and not Gaussian,
cf.~Figure~\ref{example-non-sticky-non-gaussian:fig}.
\begin{figure}
  \centering
\begin{excise}{
  \begin{tikzpicture}
    \pgfmathsetmacro{\t}{.25};
    \pgfmathsetmacro{\s}{.5};
    \pgfmathsetmacro{\ss}{(1/4)/(1+(\s+\t)/(5*\s))^2};
    \pgfmathsetmacro{\sss}{(1/4)/(1+(\t)/(5*\s))^2};
    \begin{axis}[ylabel=probability density, samples=80,
        xlabel=$y$]
       \addplot[domain=-.75:0, black] {\gauss{0}{\ss}};
       \addplot[domain=0:.75, black] {\gauss{0}{\sss}};
    \end{axis}
  \end{tikzpicture}
}\end{excise}
\includegraphics[width=0.6\textwidth]{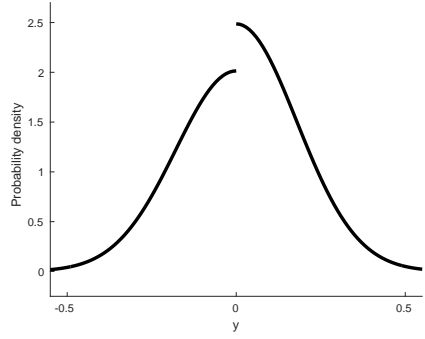}
\caption{\label{example-non-sticky-non-gaussian:fig}%
Depicted is the marginal in Example~\ref{e:Gaussian} for $t = 5$, i.e.\ the vertical component $y = \sqrt{N}e_2\cdot F_0b_N$ of the
folded empirical means multiplied by~$\sqrt{N}$.  For these, $y \to
\frac{1}{3} \mathbb{I}_{Z\geq 0} Z + \frac{2}{3} \mathbb{I}_{Z\leq 0} Z$ asymptotically in distribution as $N\to
\infty$ where $Z \sim {\mathcal N}\Big(0, 2/5\Big)$.}
\end{figure}

\end{ex}

\excise{
\begin{ex}[Non-Gaussian behavior in the nonsticky case]\label{e:Gaussian}
When the mean of a probability distribution $\mu$ is nonsticky, the
limit distribution of $\sqrt{N}\big(F_{\theta^*} b_N - (r^*,0)\big)$
on~$\Rm^2$, as described in Theorem~\ref{t:non-stickyclt}, can be
non-degenerate but fail to be Gaussian.  Let $\mu$ be the distribution
on $\KK$ with point mass $1/5$ at each of the points
\begin{align*}
  p_1 &= (r,0),\qquad p_2 = (t,\pi/2),\qquad p_3 = (s,\pi/2),\\
  p_4 &= (s,\pi), \qquad
  p_5 = (t,\alpha/2)
\end{align*}
such that $r -s -t > 0$ and $r,s,t > 0$.  Choose $s = r/2$.
Additionally, in case of $\alpha \geq 3\pi$ choose $t = r/4$ and in
case of $2\pi < \alpha < 3\pi$ set
$$%
  0 < t = \frac{r}{2}\cos(\alpha/2-\pi) <\frac{r}{2}.
$$
Then, choosing complex notation for simplicity,
$$%
m_\theta =
\begin{cases}
  \tfrac{1}{5}  (r -s)e^{-i\theta} -   \tfrac{1}{5}t &\text{for }
  0 \leq \theta \leq \alpha/2-\pi,
  \\[.5ex]
  \tfrac{1}{5} r e^{-i\theta} -\tfrac{1}{5}(s+t) &\text{for }
  -(\alpha/2-\pi)\leq \theta \leq 0.
 \end{cases}
$$
Therefore, as $\theta$ increases from $0$, the first moment $m_\theta$
starts (for $\theta=0$) at $r-s-t>0$ and moves into the lower half
plane along circular segments centered at $-t$.  As $\theta$ decreases
from $0$, the first moment $m_\theta$ moves into the upper half plane
along circular segments centered at $-s-t$.  In particular, $\mu$ is
nonsticky.  The choice of $r,s,t$ makes $m_\theta$ hit the imaginary
axis when $\theta = \theta_1 > 0$ and when $\theta = \theta_2 < 0$,
such that $t/(r-s) = \cos \theta_1$ and $(s+t)/r = \cos \theta_2$ and
such that $\theta_1$ and $\theta_2$ lie in the respective intervals
$0<\theta_1\leq \alpha/2-\pi$ and $\pi-\alpha/2 \leq \theta_2 < 0$,
where the above representation for $m_\theta$ is true.  Indeed, for
$\alpha \geq 3\pi$ this is trivial; and for $2\pi < \alpha < 3\pi$, by
hypothesis, $\frac{s+t}{r} > \frac{t}{r-s}$, which implies that
$\theta_1 > |\theta_2|$ as well as
$$%
  \theta_1 = \arccos (2t) = \frac{\alpha}{2}-\pi.
$$
Since $0 < \|m_\theta\| < m_0$ for $\theta \in (\theta_2,0) \cup
(0,\theta_1)$ with $m_{\theta_1,1} = 0 = m_{\theta_2,1}$ and
$m_{\theta,1} < 0$ for all $\theta\notin[\theta_2,\theta_1]$, it
follows that $m_0 = (r-s-t, 0)$ is the mean with
$$%
  w^+(0) = \frac{s+t}{5}\neq \frac{t}{5} = w^-(0).
$$
As a consequence of Theorem~\ref{t:non-stickyclt} and the subsequent
remarks, the limit distribution of $\sqrt{N}(F_{\theta^*} b_N -
F_{\theta^*} \bar b)$ on $\Rm^2$ is non-degenerate and not Gaussian,
cf.~Figure~\ref{example-non-sticky-non-gaussian:fig}.
\begin{figure}
  \centering
\begin{excise}{
  \begin{tikzpicture}
    \pgfmathsetmacro{\t}{.25};
    \pgfmathsetmacro{\s}{.5};
    \pgfmathsetmacro{\ss}{(1/4)/(1+(\s+\t)/(5*\s))^2};
    \pgfmathsetmacro{\sss}{(1/4)/(1+(\t)/(5*\s))^2};
    \begin{axis}[ylabel=probability density, samples=80,
        xlabel=$y$]
       \addplot[domain=-.75:0, black] {\gauss{0}{\ss}};
       \addplot[domain=0:.75, black] {\gauss{0}{\sss}};
    \end{axis}
  \end{tikzpicture}
}\end{excise}
\includegraphics[width=0.5\textwidth]{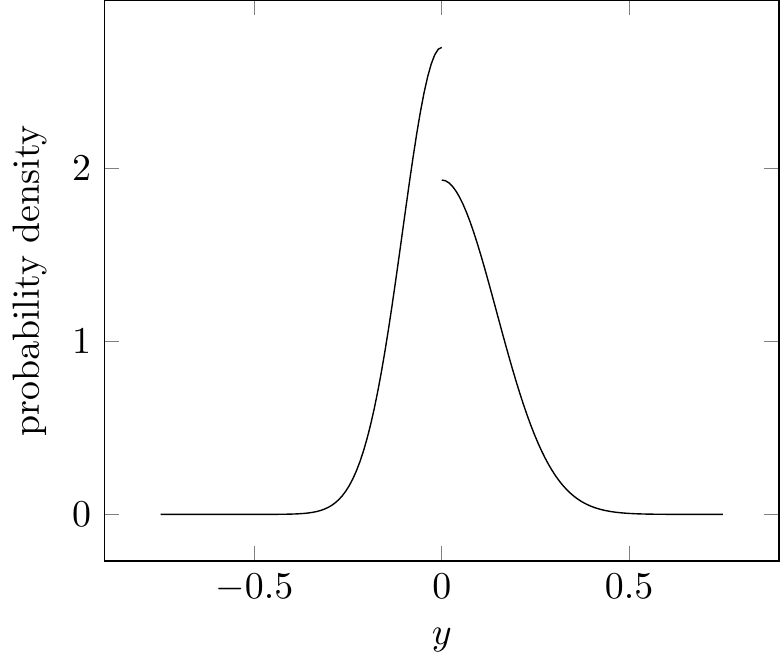}
\caption{\label{example-non-sticky-non-gaussian:fig}%
Depicted is the marginal in Example~\ref{e:Gaussian} for $r = 1 = 2s =
4t$, i.e.\ the vertical component $y = \sqrt{N}e_2\cdot F_0b_N$ of the
folded empirical means multiplied by~$\sqrt{N}$.  For these, $y \to
1_{Z\geq 0} Z + 1_{W\leq 0} W$ asymptotically in distribution as $N\to
\infty$ where $Z \sim {\mathcal
N}\Big(0,\frac{1/4}{(1+(s+t)/5(r-s-1))^2}\Big)\mbox{ and } W \sim
{\mathcal N}\Big(0,\frac{1/4}{(1+t/5(r-s-1))^2}\Big)$ because,
being binomial, ${\rm Var}(e_2\cdot Nm_0^N)=N/4$.}
\end{figure}

%
%
\end{ex}
}

\section{Folding isolated hyperbolic planar singularities}\label{s:hyperbolic}

This section elaborates on the geometric structure of the kale $\KK$
defined in~\eqref{eq:kale}.

\begin{lem}[Openness of visibility]\label{visibility-cont-property}
If $p$ is fully visible from the angle $\theta_0$ then it is fully
visible from all $\theta$ sufficiently close to~$\theta_0$.  The same
is true for invisibility.
\end{lem}
\begin{proof}
The sets $\Ical_\theta$ and $\KK\setminus
\ol{\Ical_\theta}$ are open.
\end{proof}

Recall that $d_2(z,w):\Rm^2 \times \Rm^2 \to [0,\infty)$ denotes the
Euclidean metric in $\Rm^2$.  The following lemma follows easily from
the definitions of $F_\theta$ and the metric $d$ on $\KK$.

\begin{lem}\label{l:distance-under-folding}
For any two points $p_1,p_2 \in \KK$ and any angle $\theta \in
\RR/\alpha\ZZ$,
$$%
  d_2\big( F_{\theta} (p_1), F_{\theta} (p_2)\big) \leq d(p_1,p_2),
$$
with strict inequality if $p_1 \in \Ical_\theta$, $p_2 \in \KK
\setminus \ol{\Ical_\theta}$, and $p_2$ has an angle different from $\theta$.  Moreover, for any $p\in \KK$ and $\theta \in
\RR/\alpha\ZZ$,
$$%
  d_2\big( F_{\theta} (r,\theta), F_{\theta} (p)\big) = d\big((r,\theta),p\big).
  \qedhere
$$
\end{lem}

\begin{lem}\label{lem:foldththp}
If $(1,\theta') \in \ol{\Ical_\theta}$ and $p \in \KK$, then
\begin{equation}\label{foldththp}
  e_1 \cdot F_\theta(p) \leq - e_1 \cdot F_{\theta'}(p).
\end{equation}
If $(1,\theta') \in \Ical_\theta$ then equality holds precisely when
$p \in R_{\theta,\theta'}$.
\end{lem}

\begin{rmk}\label{r:symmetric}
The conditions $(1,\theta') \in \ol{\Ical_\theta}$ and
$(1,\theta') \in \Ical_\theta$ could equivalently be expressed as
$|\theta - \theta'| \geq \pi$ and $|\theta - \theta'| > \pi$,
respectively; in particular, they are symmetric in $\theta$ and
$\theta'$.
\end{rmk}

\begin{proof}[Proof of Lemma~\ref{lem:foldththp}]
Both assertions are obvious for $p = \0$.  Hence assume $p \neq \0$,
i.e.\ that $p = (r,\hat\theta) \in \KK$ with $r > 0$.  Then
$$%
  e_1 \cdot F_\theta(p)
  =
  r \cos\big(\!\min\{|\hat\theta - \theta|,\pi\}\big),
$$
and similarly with $\theta'$ in place of~$\theta$.  The statement of
the lemma is symmetric in $\theta$ and~$\theta'$ by
Remark~\ref{r:symmetric}, so without loss of generality assume
$|\hat\theta - \theta| \geq |\hat\theta - \theta'|$.

Then $e_1 \cdot F_{\theta}(p)$ and $e_1 \cdot F_{\theta'}(p)$ are both
negative---and thus~(\ref{foldththp}) with strict inequality is
automatic---unless $|\hat\theta - \theta'| \leq \pi/2$.  Henceforth
assume $|\hat\theta - \theta'| \leq \pi/2$.  Then $|\hat\theta -
\theta| \geq \pi/2$ because $|\theta - \theta'| \geq \pi$.

If $|\hat\theta - \theta| \geq \pi$, then the left side
of~(\ref{foldththp}) is~$-r$ while the right side is
$-r\cos|\hat\theta - \theta'|$; the cosine is nonnegative because
$|\hat\theta - \theta'| \leq \pi/2$, and it achieves the value~$1$
only when $\hat\theta = \theta'$, which is when $p \in
R_{\theta,\theta'}$, as desired.

The only remaining case is where $|\hat\theta - \theta'| \leq \pi/2
\leq |\hat\theta - \theta| < \pi$.
Since $(1,\theta') \in \ol\Ical_\theta$ but $|\hat\theta - \theta| <
\pi$, the ray $\hat\theta$ must lie between $\theta'$ and~$\theta$, in
the sense that $|\theta - \theta'| = |\theta - \hat\theta| +
|\hat\theta - \theta'|$ and passing through this angle from $\theta'$
to~$\theta$ hits the ray at $\hat\theta$ along the way.  This picture
is easily drawn in the Euclidean plane~$\RR^2$, with $\theta'$ along
the horizontal axis, $\hat\theta$ in the first (northeast) quadrant,
and $\theta - \pi$ between $\theta'$ and~$\hat\theta$, possibly equal
to~$\theta'$ but never~$\hat\theta$.  (The reflection of this picture
across the horizontal axis is possible, as well, but as cosine is an
even function it changes none of the algebra.)  Using $\theta - \pi$
instead of~$\theta$ is handy because $-e_1 \cdot F_\theta(p)$ is the
cosine of the angle~$\beta$ between $\theta - \pi$ and~$\hat\theta$.
The desired result follows because $\beta \leq |\hat\theta - \theta'|
\leq \pi/2$ and
cosine is strictly decreasing on the interval $[0,\pi/2]$ while
$\beta = |\hat\theta - \theta'|$ only when $\theta - \theta' = \pi$,
which is the case $(1,\theta') \in \ol\Ical_\theta \setminus
\Ical_\theta$.\qedhere

\end{proof}

\section{Barycenters and first moments of probability measures on the kale}
\label{s:moments}

This section describes properties of the functions $\theta \mapsto
m_{\theta}$ and $\theta \mapsto m^N_{\theta}$; the behavior of these
functions aids in understanding how the barycenters $b_N$ behave in
the limit \mbox{$N \to \infty$}.  Recall that the barycenter is the
minimizer of $\Gamma(p)$, defined in~\eqref{eq:gammaDef}.  To motivate
what comes next and better explain the connection between barycenters
and the first component $m_{\theta,1}$ of folded means $m_\theta$, we
recall the analogous calculation for $\Rm^n$.  Define $\gamma\colon
\Rm^n \rightarrow [0,\infty)$ by
$$%
  \gamma(x) = \frac12\int_{\Rm^n} \| x-y\|^2 d\nu(y)
$$
for a given probability measure $\nu$ on~$\Rm^n$.  The barycenter
of~$\nu$ in this Euclidean setting is the point $x \in \Rm^n$ that
minimizes $\gamma(x)$.  Observe that
$$%
  \|x - y\|^2
  =
  \|x\|^2 - 2x\cdot y + \|y\|^2
  =
  \|x\|^2 - 2\|x\|(\hat x\cdot y) + \|y\|^2,
$$
where $\hat x = x/\|x\|$ is the unit vector in the direction of~$x$.
Hence if $\nu$ is square-integrable, and
\begin{equation}\label{eq:gtoG}
  \gamma(x) = \frac12 \|x\|^2 -\|x\| \int_{\Rm^n} (\hat x \cdot y)\,
  d\nu(y) + \gamma(0),
\end{equation}
then the minimizer of~$\gamma$ lies in the direction $\hat x$ that
maximizes
\begin{equation}\label{eq:RnIP}
  m\cdot \hat x = \int_{\Rm^n} (\hat x \cdot y)\, d\nu(y)
\end{equation}
and at a distance from the origin equal to the maximum value of
\eqref{eq:RnIP}.  Here $m \in \Rm^n$ is the mean of~$\nu$.  Hence if
$\hat x^*$ is the maximizing direction, then the barycenter can be
written in polar coordinates $(r, \hat x)$ as $(m \cdot \hat x^*, \hat
x^*)$.  From this it follows that the solution is the usual mean in
Euclidean space.  Even when the term $\gamma(0)$ in \eqref{eq:gtoG} is
infinite, it is reasonable to take this as the definition of mean.  To
make the maximization of~\eqref{eq:RnIP} well defined, one only needs
to assume $\nu$ is integrable rather than square-integrable.

A similar calculation can be done in the kale setting.  Since the
folding map rotates the direction $\theta$ back to the direction $e_1$
in the Euclidean plane, $m_{\theta,1}$ is exactly analogous to
\eqref{eq:RnIP}.  The following lemma proves the expression analogous
to \eqref{eq:gtoG} in the setting of $\KK$.

\begin{lem}\label{lem:Gammarep}
Suppose a measure~$\mu$ is square-integrable~(\ref{squareint}).  Then
for all points $(r,\theta) \in \KK$,
$$%
  \Gamma(r,\theta) = \frac{r^2}{2} - r\,m_{\theta,1} + \Gamma(\0).
$$
\end{lem}

\begin{rmk}\label{r:squareint}
As a consequence of $\|F_\theta(r,\theta)\|\leq r$, the pushforward
$\tilde \mu_\theta = \mu\circ F^{-1}_\theta$ is also square-integrable
when $\mu$ is.
\end{rmk}

\begin{proof}[Proof of Lemma~\ref{lem:Gammarep}]
For $p = (r,\theta)$, using Lemma~\ref{l:distance-under-folding},
\begin{align*}
\Gamma(p)
& = \frac{1}{2} \int_\KK d_2( F_\theta p, F_\theta q)^2 \,d\mu(q) \\
& = \frac{1}{2} \int_\KK \big(|e_1 \cdot F_\theta p - e_1 \cdot F_\theta
    q|^2 + |e_2 \cdot F_\theta p - e_2 \cdot F_\theta q|^2 \big)\,d\mu(q)\\
& = \frac{1}{2} \int_\KK \big(|r - e_1 \cdot F_\theta q|^2 + | e_2 \cdot
    F_\theta q|^2  \big)\,d\mu(q)\\
& = \frac{r^2}{2} - r \int_\KK  e_1 \cdot F_\theta q \,d\mu(q) +
    \int_\KK |\big( e_1 \cdot F_\theta q|^2 + | e_2 \cdot F_\theta q|^2 \big)
    \,d\mu(q).\qedhere
    \end{align*}
\end{proof}

Motivated by a need to understand properties of the function
$m_{\theta,1}$, we now explore its differentiability.  Define
one-sided derivatives of $g\colon \Rm/\alpha\Zm \to \Rm$ at $\theta
\in \Rm/\alpha\Zm$ by
$$%
  D^+_\theta g(\theta)
  =
  \lim_{\substack{\theta'\to\theta \\ \theta'\in(\theta,\theta+\pi)}}
  \frac{g(\theta') - g(\theta)}{\theta' - \theta}
\quad\text{and}\quad
  D^-_\theta g(\theta)
  =
  \lim_{\substack{\theta'\to\theta \\ \theta'\in(\theta-\pi,\theta)}} 
  \frac{g(\theta') - g(\theta)}{\theta' - \theta}.
$$
Recall Definition~\ref{d:interval} of the (not necessarily positive)
real number $\theta' - \theta$.  When the one-sided derivatives agree,
write $\frac{d}{d\theta} g(\theta)$ or $g'(\theta)$ as usual.

\begin{lem}\label{lem:Dmtheta}
The function $m_{\theta,1}\colon\Rm/\alpha\Zm \to \Rm$ is continuously
differentiable, and
$$%
  \frac{d}{d\theta} m_{\theta,1} = m_{\theta,2}.
$$
Moreover, for every $\theta \in \Rm/\alpha\Zm$, the one-sided
derivatives $D_\theta^\pm \frac{d m_{\theta,1}}{d \theta} =
D_\theta^\pm m_{\theta,2}$ exist and satisfy
\begin{equation}\label{Dpm-def}
  D^\pm_\theta \frac{d m_{\theta,1}}{d\theta}
  =
  D^\pm_\theta m_{\theta,2}
  =
  -m_{\theta,1} - \int_{\Ical_\theta^\mp} d(\0,p) \, d\mu(p)
  =
  -m_{\theta,1} - w^\mp(\theta)
\end{equation}
where $w^\pm(\theta)$ and $\Ical_\theta^\pm$ are as in~\eqref{wpmdef},
cf.~Figure~\ref{f:folding-map}.  In particular since $\Ical_\theta
\subset \Ical_\theta^\pm$,
\begin{equation}\label{Dpm}
  D^\pm_\theta  \frac{d m_{\theta,1}}{d\theta} \leq - m_{\theta,1} -
  \int_{\Ical_\theta} d(\0,p) \, d\mu(p)    
\end{equation}
holds for all $\theta \in \Rm/\alpha\Zm$.
\end{lem}
\begin{proof}
For $\theta' \in \Rm/\alpha\Zm$, define functions
$f_{\theta'}:\Rm/\alpha\Zm \to [-1,1]$ and $g_{\theta'}:\Rm/\alpha\Zm
\to [-1,1]$ by
$$%
  f_{\theta'}(\theta) = \cos\big(\!\min\{|\theta - \theta'|,\pi\}\big)
\quad\text{and}\quad
  g_{\theta'}(\theta)
  =
  \begin{cases}
  \\[-4.25ex]
  \sin(\theta'-\theta) & \text{if } |\theta-\theta'| \leq \pi
  \\
     \hfill 0 \hfill   & \text{otherwise}.
  \\[-.9ex]
  \end{cases}
$$
Then
\begin{equation}\label{m-repr:eq}
  m_{\theta,1} = \int_\KK r_p \, f_{\theta_p}(\theta) \,d\mu(p)
\quad\text{and}\quad
  m_{\theta,2} = \int_\KK r_p\, g_{\theta_p}(\theta) \,d \mu(p),
\end{equation}
where $p = (r_p,\theta_p)$.  Each function $f_{\theta_p}$ is
continuously differentiable.  The integrability
condition~(\ref{integ}) and the dominated convergence theorem imply
that
$$%
  \frac{d m_{\theta,1}}{d\theta} =  \int_\KK r_p \,
  f'_{\theta_p}(\theta) d\mu(p) = \int_{\KK \setminus
  \Ical_\theta} r_p \, \sin(\theta_p - \theta) d\mu(p) =
  m_{\theta,2}.
$$
Each $g_{\theta_p}$ has one-sided derivatives:
\begin{align*}
  D^+_\theta g_{\theta_p}(\theta) =&
  \begin{cases}
  \\[-4.25ex]
  -\cos(\theta_p - \theta)
    &\quad\text{if } \theta_p - \pi \leq \theta < \theta_p + \pi,
     \text{ i.e.\ } -\pi < \theta_p - \theta \leq \pi
  \\
  \hfill 0 \hfill
    &\quad\text{otherwise},
  \\[-.9ex]
  \end{cases}
\\[1ex]
 D^-_\theta g_{\theta_p}(\theta) =&
 \begin{cases}
 \\[-4.25ex]
  -\cos(\theta_p - \theta)
    &\quad\text{if } \theta_p - \pi < \theta \leq \theta_p + \pi
     \text{ i.e.\ } -\pi \leq \theta_p - \theta <\pi
  \\
  \hfill 0 \hfill
    &\quad\text{otherwise}.
  \\[-.9ex]
 \end{cases}
\end{align*}
Therefore, by the dominated convergence theorem, $m_{\theta,2}$ also
has one-sided derivatives at every $\theta \in \Rm/\alpha\Zm$:
\begin{align*}
D^+_\theta m_{\theta,2}
  =  \int_\KK r_p \,D^+_\theta g_{\theta_p}(\theta) \,d\mu(p)
&  = - \int_{\KK \setminus \Ical_\theta^-} r_p \cos(\theta_p -
   \theta) \,d\mu(p)\\
 & = - m_{\theta,1} - \int_{\Ical_\theta^-} r_p \, d\mu(p).
\end{align*}
Similarly,
\begin{align*}
D^-_\theta m_{\theta,2}
  =  \int_\KK r_p \,D^-_\theta g_{\theta_p}(\theta) \,d\mu(p)
&  = - \int_{\KK \setminus \Ical_\theta^+} r_p \cos(\theta_p -
   \theta) \,d\mu(p)
\\
&  = - m_{\theta,1} - \int_{\Ical_\theta^+} r_p \, d\mu(p).
\end{align*}
In particular, (\ref{Dpm}) holds for all $\theta \in \Rm/\alpha\Zm$.
\end{proof}

\begin{cor}\label{cor:noshadowmass}
Let $A \neq B$ and $|A - B| \leq \pi$.  If $m_{\theta,1} = 0$ for all
$\theta \in [A,B]$ then $\mu(\Ical_\theta) = 0$ for all $\theta \in
[A,B]$.
\end{cor}
\begin{proof}
For $\theta \in (A,B)$ this an immediate consequence of~(\ref{Dpm}),
since
$$%
  0 = D^\pm_\theta  \frac{d m_{\theta,1}}{d\theta} \leq -
  \int_{\Ical_\theta} d(\0,p) \, d\mu(p) \leq 0.
$$
When $0 < B - A \leq \pi$, the $D^+_\theta$ and $D^-_\theta$ versions
of this calculation remain valid for the endpoints $\theta = A$ and
$\theta = B$, respectively, and swapped when $0 < A - B \leq \pi$.
\end{proof}

\begin{ex}\label{e:noshadowmass}
The assertion of Corollary~\ref{cor:noshadowmass} is wrong when $A =
B$, i.e.~when $\theta^* \in \Rm/\alpha\Zm$ with $m_{\theta^*,1} = 0$
is isolated.  To see this, consider $\mu$ having point masses of
weight $1/3$ at $(1,\theta^*)$ as well as at $(1/2, \theta^* + \pi +
\epsilon)$ and $(1/2, \theta^* - \pi - \epsilon)$ with $0 < \epsilon <
\alpha/2 - \pi$.  Then $\mu(\Ical_{\theta^*}) = 2/3$ while
$m_{\theta,1} < 0$ for all $\theta \neq \theta^*$ and $m_{\theta^*,1}
= 0$.
\end{ex}

\begin{ex}\label{e:pos-shadowmass}
The shadow of an angle $\theta$ with $m_{\theta,1} > 0$ may carry
mass.  Changing the first point in Example~\ref{e:noshadowmass} into
$(2,\theta^*)$ yields $m_{\theta^*,1} =1/3 >0$ and
$\mu(\Ical_{\theta^*}) = 2/3$.
\end{ex}

Recalling the definition $w^\pm(\theta)$ from~\eqref{wpmdef}, observe
that $0 \leq w^\pm(\theta) \leq \bar r$ holds for all~$\theta$ because
the integrand is nonnegative and $\Ical^\pm_\theta \subset \KK$.
Also, as a consequence of~(\ref{m-repr:eq}),
\begin{equation}\label{m-rbar:ineq}
  \|m_\theta\|
  =
  \sqrt{m_{\theta,1}^2+m_{\theta,2}^2}
  \leq
  \sqrt{2}\bar r.
\end{equation}

Since $\mu$ is a probability measure, due to $\sigma$-additivity, only
countably many of the rays $\{(r,\theta) \mid 0 \leq r < \infty\}$ for
$\theta \in \RR/\alpha\ZZ$ carry positive mass of $\mu$.  Consequently,
$w^+$ and $w^-$ are continuous almost everywhere with
respect to the understood measure on $\RR/\alpha\ZZ$ induced by
Lebesgue measure on $[0,\alpha)$, and so is $\theta \mapsto D^\pm
m_{\theta,2}$. Furthermore, $w^+(\theta) = w^-(\theta)$ for almost every angle $\theta$.  

\begin{cor}\label{cor:odeintegrate}
Let $\hat\theta \in \Rm/\alpha\Zm$ and $\theta \in [\hat\theta - \pi,
\hat\theta + \pi]$.  Then
\begin{equation}\label{eq:m_theta1}
  m_{\theta,1} = m_{\hat \theta,1} \cos(\theta - \hat \theta) +
  m_{\hat \theta,2} \sin(\theta - \hat \theta) - \int_{\hat
  \theta}^\theta w(\psi) \sin(\theta - \psi) \,d\psi
\end{equation}
and
\begin{equation}\label{eq:m_theta2}
  m_{\theta,2} = - m_{\hat \theta,1} \sin(\theta - \hat \theta) +
  m_{\hat \theta,2} \cos(\theta - \hat \theta)- \int_{\hat
  \theta}^\theta w(\psi) \cos(\theta - \psi) \,d\psi.
\end{equation}
where 
\begin{equation}
w(\psi) = \int_{\Ical_\psi} d(\0, p) \,d\mu(p). \label{wpmdef2}
\end{equation}
\end{cor}
\begin{proof}
Since $w^+$ and $w^-$ are equal for almost every angle, we have 
\[
\int_{\hat
  \theta}^\theta w^\pm(\psi) \sin(\theta - \psi) \,d\psi = \int_{\hat
  \theta}^\theta w(\psi) \sin(\theta - \psi) \,d\psi
\]
and similarly for the integral in (\ref{eq:m_theta2}). Equation (\ref{eq:m_theta1}) then follows from~(\ref{Dpm-def}) using integration by
parts along $\psi\in (\hat \theta,\theta)$:
$$%
  m_{\psi,1}\sin(\theta-\psi)
  =
  -D^{\pm}_\psi  m_{\psi,2}\sin(\theta-\psi) - w^\mp(\psi) \sin(\theta-\psi).
  $$
Equation (\ref{eq:m_theta2}) follows from
$$%
\hspace{11.6ex}
  D^{\pm}_\psi  m_{\psi,2}\cos(\theta-\psi)
  =
  - m_{\psi,1}\cos(\theta-\psi) - w^\mp(\psi) \cos(\theta-\psi).
$$
\end{proof}

\begin{figure}[t!]
\centering

\begin{minipage}{0.25\linewidth}
\begin{excise}{
\begin{tikzpicture}[scale=3]
  \coordinate (O) at (0,0);
  \coordinate (theta) at (0:1);
  \coordinate (thetaP) at (-30:1);
  \coordinate (m) at (45:.75);

  \draw[solid] (theta) node[right]{$\theta$}--(O)--(thetaP)  node[right]{$\hat\theta$};
  \draw[dashed] (O)--(m);
  \fill (m) circle[radius=.5pt] node[right]{$m_\theta$};

  \draw pic[draw=orange,arrows={-latex[orange, fill=orange]},solid,thick,angle
        eccentricity=1.05, angle radius=1.25cm,right,"$\sigma=\theta-\hat\theta$"]
	{angle=thetaP--O--theta};

  \draw[opacity=0,white] (0,-.75) rectangle (1,1);
\end{tikzpicture}
}\end{excise}
\includegraphics[width=1\textwidth]{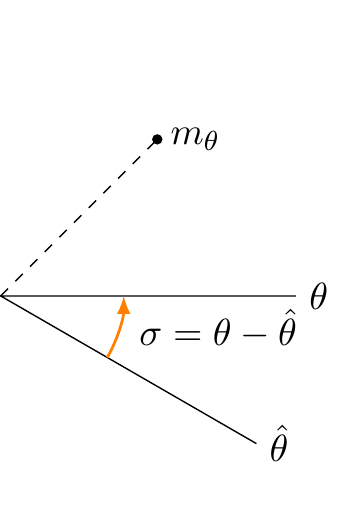}
\end{minipage}
\begin{minipage}{0.225\linewidth}
$$%
  \xrightarrow{\quad\Phi_\sigma \quad}
$$
\end{minipage}
\begin{minipage}{0.25\linewidth}
\begin{excise}{
\begin{tikzpicture}[scale=3]
  \coordinate (O) at (0,0);
  \coordinate (theta) at (30:1);
  \coordinate (thetaP) at (0:1);
  \coordinate (m) at (45:.75);
  \coordinate (mNew) at (75:.75);
  \coordinate (thetaPOld) at (-30:1);

  \draw[solid] (theta) node[right]{$\theta$}--(O)--(thetaP) node[right]{$\hat\theta$};

  \draw[white] (O)--(thetaPOld);

  \draw[dashed] (O)--(m);
  \fill (m) circle[radius=.5pt] node[right]{$m_\theta$};

  \draw[dashed] (O)--(mNew);
  \fill (mNew) circle[radius=.5pt] node[above]{$\Phi_\sigma m_\theta$};

  \draw pic[draw=orange,arrows={-latex[orange, fill=orange]},solid,thick,angle
        eccentricity=1.05, angle radius=1.25cm,right,"$\sigma$"]
        {angle=thetaP--O--theta};

  \draw pic[draw=orange,arrows={-latex[orange, fill=orange]},solid,thick,angle
        eccentricity=1.05, angle radius=1.55cm,above,"$\sigma$"]
        {angle=m--O--mNew};
  \draw[opacity=0,white] (0,-.75) rectangle (1,1);
\end{tikzpicture}
}\end{excise}
\includegraphics[width=1\textwidth]{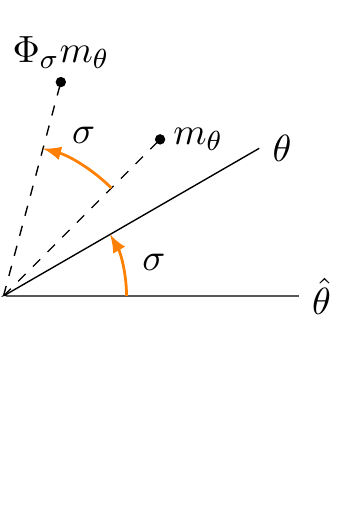}
\end{minipage}
\caption{\label{f:rot}%
The rotation.  If all the shadows from $\theta$ to $\hat\theta$ carry
no mass then $\Phi_{\hat\theta-\theta}m_\theta = m_{\hat\theta}$ by
Lemma~\ref{lem:mthetamodulus}.}
\end{figure}

For an angle $\sigma \in \Rm$ define the rotation $\Phi_\sigma:\Rm^2
\to \Rm^2$ by $\Phi_\sigma (r,\psi) = (r,\psi + \sigma)$ in polar
coordinates, cf.~Figure~\ref{f:rot}.  As usual, denote by $\|\cdot\|$
the standard Euclidean norm on~$\RR^2$.  Recall
Definition~\ref{d:IcalA}, specifically $\Ical_A$ for an interval~$A$.

\begin{lem}\label{lem:mthetamodulus}
Let $\hat\theta \in \Rm/\alpha\Zm$.  For all $\theta \in
\Rm/\alpha\Zm$ with $|\theta - \hat\theta| \leq \pi$,
$$%
  \|\Phi_{\hat\theta - \theta} m_{\hat\theta} - m_\theta\|
  \leq
  |\theta - \hat\theta| \int_{\Ical_{(\hat\theta,\theta)}} d(\0,p)\,d\mu(p).
$$
In particular, $\|\Phi_{\hat\theta - \theta} m_{\hat\theta} -
m_\theta\| \leq |\theta - \hat\theta| \bar r$.  Also, if
$\mu(\Ical_{(\hat\theta,\theta)}) = 0$, then $\Phi_{\hat\theta -
\theta} m_{\hat\theta} = m_\theta$.
\end{lem}
\begin{proof}
Suppose $\theta \in [\hat\theta-\pi, \hat\theta + \pi]$.  Then
$$%
  \Phi_{\hat\theta - \theta} m_{\hat\theta} =
  \begin{pmatrix}
    \phantom{-}\cos(\theta - \hat\theta) & \sin(\theta - \hat\theta)\\
              -\sin(\theta - \hat\theta) & \cos(\theta - \hat\theta)
  \end{pmatrix}
  \begin{pmatrix}
  	m_{\hat\theta,1}\\
	m_{\hat\theta,2}
  \end{pmatrix}.
$$
So, by Corollary~\ref{cor:odeintegrate},
\begin{align*}
  \|\Phi_{\hat\theta - \theta} m_{\hat\theta} - m_\theta\|^2
& =
  \left|\int_{\hat\theta}^\theta w(\psi)\sin(\theta - \psi)\,d\psi\right|^2
  +\
  \left|\int_{\hat\theta}^\theta w(\psi)\cos(\theta - \psi)\,d\psi\right|^2
\\
& \leq
  |\theta-\hat\theta|\int_{\hat\theta}^\theta w(\psi)^2\sin^2(\theta-\psi)\,d\psi
  +
  |\theta-\hat\theta|\int_{\hat\theta}^\theta w(\psi)^2\cos^2(\theta-\psi)\,d\psi
\\
& =
  |\theta-\hat\theta|\int_{\hat\theta}^\theta w(\psi)^2\,d\psi
  \leq
  |\theta-\hat\theta|^2\sup_{\psi\in(\hat\theta,\theta)}w(\psi)^2.
\end{align*}
The assertion follows now from
$$%
\hspace{15.8ex}
  \sup_{\psi \in (\theta, \hat\theta)} w(\psi)
  =
  \sup_{\psi\in(\theta,\hat\theta)}\int_{I_\psi} d(\0,p)\,d\mu(p)
  \leq
  \int_{\Ical_{(\theta,\hat\theta)}} d(\0,p)\,d\mu(p).
\makebox[15.8ex][r]{\qedhere}
$$
\end{proof}

\begin{lem}\label{lem:mshadow}
If $\theta,\theta'\in \Rm/\alpha\Zm$ with $(1,\theta') \in
\ol{\Ical_\theta}$, then $m_{\theta,1} \leq - m_{\theta',1}$.
If $(1,\theta') \in \Ical_\theta$, then $m_{\theta,1} =
-m_{\theta',1}$ if and only if $\mu(R_{\theta,\theta'}) = 1$.
\end{lem}
\begin{proof}
This follows from Lemma~\ref{lem:foldththp} and the definition of
$m_\theta$ and~$m_{\theta'}$:
\begin{align*}
m_{\theta,1}
& = \int_{R_{\theta,\theta'}}\! e_1 \cdot F_\theta (p) \, d\mu(p) +
  \int_{\KK \setminus R_{\theta,\theta'}}\! e_1 \cdot F_\theta( p) \,
  d\mu(p)
\\
& = - \int_{R_{\theta,\theta'}}\! e_1 \cdot F_{\theta'} (p) \, d\mu(p) +
  \int_{\KK \setminus R_{\theta,\theta'}}\! e_1 \cdot F_\theta (p) \,
  d\mu(p) \text{ by the equality case in Lemma~\ref{lem:foldththp}}
\\
& \leq - \int_{R_{\theta,\theta'}}\! e_1 \cdot F_{\theta'} (p) \,
  d\mu(p) - \int_{\KK \setminus R_{\theta,\theta'}}\! e_1 \cdot
  F_{\theta'} (p) \, d\mu(p) \text{ by Lemma~\ref{lem:foldththp}}
\\
& = -m_{\theta',1}.
\end{align*}
If $(1,\theta') \in \Ical_\theta$, then equality holds only if
$\mu(\KK \setminus R_{\theta,\theta'}) = 0$.
\end{proof}

\begin{cor}\label{cor:mshadow2}
The nondegeneracy condition (Assumption~\ref{d:nondegen}) implies that
$m_{\theta,1} < - m_{\theta',1}$ whenever $(1,\theta') \in
\Ical_\theta$, or in other words, whenever $|\theta - \theta'| > \pi$.
\end{cor}

\begin{prop}\label{p:cases}
Assuming integrability (\ref{integ}) and nondegeneracy
(Definition~\ref{d:nondegen}), the subset of\/ $\RR/\alpha\ZZ$ on
which $m_{\theta,1} \geq 0$ is a closed interval that is exactly one
of the following:
\begin{itemize}
\item[\it(i)]%
empty,
\item[\it(ii)]%
of length $< \pi$, with $m_{\theta,1} = 0$ on its entirety,~or
\item[\it(iii)]%
of length $\leq \pi$, with $m_{\theta,1}$ strictly concave (and hence
strictly positive) on its interior.
\end{itemize}
The length of the interval depends on~$\mu$ as well as on~$\alpha$.
\end{prop}

\begin{proof}
In any case, Corollary~\ref{cor:mshadow2} implies that $\min
m_{\theta,1} < 0$.  Henceforth assume case (i) does not hold, so the
set $K$ of points where $m_{\theta,1} \geq 0$ is nonempty.  Because
$m_{\theta,1}$ is continuous, the subset $K \subset \RR/\alpha\ZZ$ is
closed and $m_{\theta,1} = 0$ on its boundary.

First suppose that $\max m_{\theta,1} > 0$, the goal being to reach
conclusion~(iii).  Then $K$ contains distinct points $A$ and~$B$ where
$m_{\theta,1} = 0$.  Corollary~\ref{cor:mshadow2} implies that $|B -
A| \leq \pi$.  Lemma~\ref{lem:Dmtheta} implies that $m_{\theta,1}$ is
strictly concave whenever $m_{\theta,1} > 0$.  Hence we can and do
assume that
$$%
  m_{B,2}
  =
  \frac{d~}{d \theta}\Big|_{\theta = B} m_{\theta,1}
  <
  0
  <
  \frac{d~}{d \theta}\Big|_{\theta = A} m_{\theta,1}
  =
  m_{A,2},
$$
using Lemma~\ref{lem:Dmtheta} and the fact that $m_{\theta,1} > 0$
whenever $\theta \in (A,B)$.  Now (\ref{eq:m_theta1}) implies that
\begin{equation}\label{mintegrate}
m_{\theta,1} =
\begin{cases}
\\[-4ex]
  \displaystyle
  m_{A,2}\sin(\theta-A) - \int_A^\theta w(\psi)\sin(\theta-\psi)\,d\psi
& \text{ if } \theta \in (A - \pi,A)
\\[2ex]
  \displaystyle
  m_{B,2}\sin(\theta-B) - \int_B^\theta w(\psi)\sin(\theta-\psi)\,d\psi
& \text{ if } \theta \in (B,B + \pi),
\\[-.5ex]
\end{cases}
\end{equation}
and both of these are strictly negative.  Since also $m_{\theta,1} <
0$ for all $\theta \notin (\theta'-\pi, \theta'+\pi)$ for all $\theta'
\in [A,B]$ by Corollary~\ref{cor:mshadow2}, conclusion (iii) follows
when $\max m_{\theta,1} > 0$.

Finally, assume $\max m_{\theta,1} = 0$.  Fix a left boundary
point~$A$ a right boundary point $B$ of~$K$; note that $A = B$ is
possible.  Corollary~\ref{cor:mshadow2} again teaches that $B - A \leq
\pi$.  By~hypothesis,
$$%
  m_{B,2}
  =
  \frac{d~}{d \theta}\Big|_{\theta = B} m_{\theta,1}
  =
  0
  =
  \frac{d~}{d \theta}\Big|_{\theta = A} m_{\theta,1}
  =
  m_{A,2}.
$$
Hence (\ref{eq:m_theta1}) takes the forms
$$%
  m_{\theta,1} =  \int^A_{\theta} w(\psi) \sin(\theta - \psi) \,d\psi
  \quad\text{and}\quad
  m_{\theta,1} = -\int_B^{\theta} w(\psi) \sin(\theta - \psi) \,d\psi.
$$
These formulas, plus the choices of $A$ and $B$ as left and right
endpoints, imply that $m_{\theta,1} < 0$ for all $\theta \in [A - \pi,
A) \cup (B, B + \pi]$.  In words, every left endpoint of~$K$ is
preceded by, and every right endpoint of~$K$ is followed by, an
interval of length at least~$\pi$ on which $m_{\theta,1} < 0$.  Since
$|A - B| \leq \pi$, the interval $[A,B]$ contains no endpoints of~$K$
other than $A$ and~$B$ themselves.  Therefore $m_{\theta,1} = 0$ for
all $\theta \in [A,B]$.  Corollary~\ref{cor:mshadow2} prevents
$m_{\theta,1} \geq 0$ for $\theta$ outside of $[A - \pi, B + \pi]$.
Except for showing the strict inequality $|B - A| < \pi$, this
completes the proof that $\max m_{\theta,1} = 0$ forces
conclusion~(ii).

Suppose, then, that $|B - A| = \pi$.  Corollary~\ref{cor:noshadowmass}
implies that $\mu(C_{[A,B]}) = 1$.  If $\theta^*$ is the midpoint of
the interval $[A,B]$, then the measure $\tilde\mu_{\theta^*}$ is
supported on the half-space $H^+ = \{(z_1,z_2) \in \RR^2 \mid z_1 \geq
0\}$.  But $m_{1,\theta} = 0$ for all $\theta \in [A,B]$, whence
$\tilde\mu_{\theta^*}$ is actually supported on a single line
$\partial H^+$.  This contradicts the non-degeneracy hypothesis.
Therefore $|B - A| < \pi$, as desired.
\end{proof}

\begin{cor}\label{cor:thetastar}
If $\max_\theta m_{\theta,1} > 0$, then there is a unique angle
$\theta^*$ at which the maximum is attained: $m_{\theta^*,1} =
\max_\theta m_{\theta,1}$.  Furthermore, $m_{\theta^*,2} = 0$ for that
angle.
\end{cor}
\begin{proof}
The claim concerning $m_{\theta^*,1}$ is immediate from the concavity
in Proposition~\ref{p:cases}.  The fact that $m_{\theta^*,2} = 0$
follows from the first claim of Lemma~\ref{lem:Dmtheta}: $m_{\theta,2}
= \frac{d}{d\theta} m_{\theta,1}$.
\end{proof}

\begin{cor}\label{cor:Gammamax}
Assume square-integrability (\ref{squareint}).  If $\max_\theta
m_{\theta,1} \leq 0$ then $\Gamma(p)$ attains its minimum at the
unique point $\bar b = \0$.  If $\max_\theta m_{\theta,1} > 0$, then
$\Gamma(p)$ attains its minimum at the unique point $\bar b =
(m_{\theta^*,1}, \theta^*)$, where
$$%
  \theta^* = \argmax_\theta m_{\theta,1}.
$$
In either case, the mean of $\mu$ in Definition~\ref{d:meanK}
coincides with the barycenter of~$\mu$.
\end{cor}
\begin{proof}
Use the explicit expression for $\Gamma(p)$ from
Lemma~\ref{lem:Gammarep} and Corollary~\ref{cor:thetastar}; minimize
over $r$ and~$\theta$.
\end{proof}

\begin{cor}\label{cor:bNetaN}
Assume square-integrability (\ref{squareint}).  If there is
$(r',\theta')\in \KK$ with $r'\geq 0$ and $m_{\theta'} = (r',0)$ then
$\bar b = (r',\theta')$.
\end{cor}
\begin{proof}
When $r' = m_{\theta',1} > 0$, Proposition~\ref{p:cases}(iii)
holds, and $\theta'$ lies interior to the closed interval $[A,B]$
there.  Due to Corollary\ref{cor:Gammamax}, $m_{\theta,1}$ attains a
unique maximum at $\theta^* \in (A,B)$.  Moreover, $m_{\theta'} =
(r',0)$ and $m_{\theta',2} = 0$ imply, with Lemma~\ref{lem:Dmtheta},
that
\begin{equation}\label{concave1}
  \frac{d}{d \theta} m_{1,\theta} \Big|_{\theta = \theta'} =
  m_{\theta',2} = 0.
\end{equation}
By strict concavity in Proposition~\ref{p:cases}(iii), $\theta^* =
\theta'$, so $\bar b = (r',\theta')$ by Corollary~\ref{cor:Gammamax}.

The case $r' = m_{\theta',1} = 0$ can only occur in cases (ii)
and~(iii) of Proposition~\ref{p:cases}, with $\theta'$ being an
endpoint of the closed interval in case (iii) and anywhere in the
closed interval in case~(ii).  Since (\ref{concave1}) holds
nonetheless, strict concavity in case (iii) cannot be.  Consequently,
$m_{\theta,1} \leq 0$ for all~$\theta$.  Therefore, by
Corollary~\ref{cor:Gammamax}, $\bar b = (0,\theta') = (r',\theta')$.
\end{proof}

We conclude this section with important estimates relating folded
averages $\eta_{\theta,N}$ from~(\ref{etaNdef}) to folded barycenters
$F_\theta b_N$ of empirical distributions on~$\KK$.

\begin{lem}\label{lem:bNetaN2}
Suppose that $b_N = (\hat r, \hat\theta)$ with $\hat r > 0$.  If
$\theta \in \Rm/\alpha\Zm$ and $|\theta - \hat\theta| \leq \pi$, then
$$%
  \|\eta_{\theta,N} - F_\theta b_N\|
  \leq
  \frac{|\theta-\hat\theta|}{N}\sum_{p_n\in\Ical_{(\theta,\hat\theta)}} d(\0,p_n).
$$
In particular, $\eta_{\hat\theta,N} = F_{\hat\theta} b_N$.  Also, if
$p_n \notin \Ical_{(\theta,\hat\theta)}$ for all $n = 1,\ldots,N$,
then $\eta_{\theta,N} = F_\theta b_N$.
\end{lem}
\begin{proof}
This is a consequence of Lemma~\ref{lem:mthetamodulus} applied to the
measure $\mu^N$ and the associated first moments $m^N_\theta =
\eta_{N,\theta}$.  By Corollary~\ref{cor:Gammamax}, $b_N =
(m_{\hat\theta,1}^N, \hat\theta)$ and $m_{\hat \theta,2}^N = 0$
because $\hat r > 0$.  Hence $F_{\hat\theta} b_N = (m^N_{\hat
\theta,1},0) = m^N_{\hat\theta}$.  As $|\theta - \hat\theta |\leq
\pi$, in polar coordinates $F_\theta b_N = (\hat r, \hat\theta -
\theta) = \Phi_{\hat\theta - \theta} F_{\hat\theta}b_N$,~so
$$%
  \|F_{\theta} b_N - \eta_{\theta,N}\|
  =
  \|\Phi_{\hat\theta - \theta} F_{\hat\theta} b_N - \eta_{\theta,N}\|
  =
  \|\Phi_{\hat\theta - \theta} m^N_{\hat\theta} - m_\theta^N\|.
$$
Therefore, by Lemma~\ref{lem:mthetamodulus},
$$%
\hspace{11.5ex}
  \|F_{ \theta}b_N - \eta_{\theta,N}\| \leq |\theta - \hat\theta|
  \int_{\Ical_{(\theta, \hat\theta)}} d(\0,p) \mu_N(p)
  =
  \frac{|\theta - \hat\theta|}{N} \sum_{p_n \in \Ical_{(\theta,
  \hat\theta)}} d(\0,p_n).
\makebox[11.5ex][r]{\qedhere}
$$
\end{proof}

The following is a special version of Corollary~\ref{cor:bNetaN}.

\begin{cor}\label{cor:bNetaN3}
If $\eta_{\theta',N} = (r',0) \in \Rm^2$ with $r' \geq 0$, then $b_N =
(r',\theta')$.
\end{cor}

\section{Proof of the sticky law of large numbers}\label{s:slln}

The standard law of large numbers for folded averages in $\Rm^2$
states that $m_\theta^N \rightarrow m_\theta$ as $N \rightarrow
\infty$.  It holds uniformly in~$\theta$, as follows.

\begin{lem}\label{l:mconv}
For any $\epsilon > 0$, there is a random integer
$N^*_\epsilon(\omega)$ such that $\displaystyle\max_{\theta \in
\ZZ/\alpha \ZZ} \|m_\theta^N - m_\theta\| \leq \epsilon$ for all $N
\geq N^*_\epsilon(\omega)$, $\Pm\text{ almost surely}$.
\end{lem}
\begin{proof}
Fix $\epsilon > 0$ and an integer $n > \max\big(24 \alpha \bar
r/\epsilon, \alpha/\pi\big)$, and let $\theta_k = \alpha k/n +
\alpha\ZZ$ for $k = 0,\ldots,n-1$.  Then $|\theta_k - \theta_{k+1}| =
\alpha/n<\pi$.  For any $\theta \in [\theta_k,\theta_{k+1})$,
$$%
  \|\Phi_{\theta_k - \theta}z - z\|
  \leq
  \frac{\alpha\|z\|}{n} \text{ for any } z \in \RR^2
$$
as well as by Lemma~\ref{lem:mthetamodulus},
\begin{equation}\label{rotated-moment:ineq}
  \|\Phi_{\theta_k - \theta} m_{\theta_k} - m_\theta\|
  \leq
  \frac{\alpha\bar{r}}{n}.
\end{equation}
Hence, making also use of~(\ref{m-rbar:ineq}),
\begin{equation}\label{mthetak-mtheta:ineq}
  \|m_{\theta_k} - m_\theta\|
  \,\leq\,
  \|m_{\theta_k} - \Phi_{\theta_k - \theta} m_{\theta_k}\| +
  \|\Phi_{\theta_k - \theta} m_{\theta_k} - m_\theta\|
  \,\leq\,
  4~\frac{\alpha\bar{r}}{n}.
\end{equation}
By the law of large numbers~(\ref{moments:slln}), there is nullset
${\mathcal M}_1$ and an integer $N_1(\omega)$ such that
$\|m^N_{\theta_k} - m_{\theta_k}\| \leq \epsilon/3$ for all $N \geq
N_1(\omega)$, all $k \in \{0,\ldots,n-1\}$, and all $\omega \in \Omega
\setminus {\mathcal M}_1$.  Similarly, by the law of large numbers
there is also a nullset ${\mathcal M}_2$ and an integer $N_2(\omega)$
such that
$$%
  0 \leq \ol{r^N}:=\int_\KK d(\0,p) \,d\mu_N(p) \leq 2 \bar r
$$
for all $N \geq N_2(\omega)$ and all $\omega \in \Omega \setminus
{\mathcal M}_2$.  Applying (\ref{mthetak-mtheta:ineq}) to the
empirical moments gives thus
$$%
  \|m^N_{\theta_k} - m^N_\theta\|
  \leq
  \frac{4\alpha\ol{r^N}}{n}
  \leq
  8\frac{\alpha\bar{r}}{n}
$$
for all $N \geq N_2(\omega)$.  Finally,
\begin{align*}
\|m^N_\theta - m_\theta\|
& \leq \|m^N_\theta - m^N_{\theta_k}\| + \|m^N_{\theta_k} -
  m_{\theta_k}\|  + \|m_{\theta_k} - m_\theta\|
\\
& \leq 8\frac{\alpha \bar r}{n} + \frac{\epsilon}{3} +
  4\frac{\alpha\bar r}n
  <
  \epsilon
\end{align*}
for all $N \geq N^*_\epsilon(\omega) = \max(N_1(\omega),N_2(\omega))$
and $\Omega\ni\omega \not\in {\mathcal M}={\mathcal M}_1\cup {\mathcal
M}_2$.
\end{proof}

Given a set of angles $T \subset \Rm/\alpha\Zm$, define the set
\begin{equation}\label{ATp}%
  C_T^+ = \{(r,\theta) \in \KK \mid r > 0 \text{ and } \theta \in T\}
  = C_T \setminus \{\0\},
\end{equation}
which excludes the origin from the sector~$C_T$
(Definition~\ref{d:sector}).

\begin{thm}\label{t:stickyLLN}
Let $T \subset \Rm/\alpha\Zm$ be a closed subset such that
$m_{\theta,1} < 0$ for all $\theta \in T$.  Then there is a random
integer $N^*(\omega)$ such that
$$%
  b_N(\omega) \notin C_T^+ \text{ for all } N \geq N^*(\omega)
$$
holds $\Pm$-almost surely.  In particular, if $\mu$ is fully sticky
then there is a random integer $N^*(\omega)$ such that $b_N = \0$ for
all $N \geq N^*(\omega)$, $\Pm$-almost surely.  Similarly, if $\mu$ is
partly sticky and $T \subset \Rm/\alpha\Zm$ is any open interval
containing the maximal interval where $m_{\theta,1} = 0$, as described
in Propositions~\ref{p:casesSimple}~and~\ref{p:cases}, then $b_N
\in C_T$ for all $N \geq N^*(\omega)$, $\Pm$-almost surely.
\end{thm}
\begin{proof}
Since $T$ is closed and $m_{\theta,1}$ is continuous, there is
$\epsilon > 0$ such that $\sup_{\theta \in T} m_{\theta,1} < -\epsilon
< 0$.  By Lemma~\ref{l:mconv} there is a random integer $N^*(\omega)$
such that $m^N_{\theta,1} < -\epsilon/2$ for all $\theta \in T$,
almost surely for all $N \geq N^*(\omega)$.  Now, $b_N$ is the unique
minimizer of
$$%
  p \mapsto \Gamma_N(p) = \frac{1}{2} \int_\KK d(p,q)^2 \,d \mu_N(q).
$$
Since the empirical measures $\{\mu_N\}_{N=1}^\infty$ are
square-integrable (even if $\mu$ is not),
\begin{equation}\label{eq:Gammaemp}
  \Gamma_N(r,\theta) = \frac{r^2}{2} - r\,m^N_{\theta,1} + \Gamma_N(0)
\end{equation}
by Lemma~\ref{lem:Gammarep}.  Therefore, if $\theta \in T$, and $r >
0$, and $N \geq N^*(\omega)$, then almost surely
$$%
  \Gamma_N(r,\theta)
  >
  \frac{r^2}{2} +\frac{\epsilon}{2} r + \Gamma_N(0) > \Gamma_N(0).
$$
Hence the minimizer $b_N$ lies outside of $C_T^+$ almost surely.
\end{proof}

By a very similar argument, Corollary~\ref{cor:Gammamax} and
Lemma~\ref{l:mconv} together imply the following, which we state
without proof.  It also is a consequence of the strong law of
\cite{Z77}.

\begin{thm}\label{t:llnAint}
Suppose that $\max_\theta m_{\theta,1} = m_{\theta^*,1} > 0$.  Let $T
\subset \Rm/\alpha\Zm$ be any open interval of length $\leq \pi$
containing~$\theta^*$.  There is a random integer $N^*(\omega)$ such
that
$$%
  b_N(\omega) \in C_T^+ \text{ for all } N \geq N^*(\omega)
$$
holds $\Pm$-almost surely.  In particular, if $\mu$ is nonsticky then
for any $\epsilon \in (0,\pi/2)$, the empirical barycenter $b_N$ lies
in $C_{(\theta^*-\epsilon,\theta^*+\epsilon)}^+$ for all $N
>N^*(\omega)$, $\Pm$-almost surely.
\end{thm}

We now give the proof the law of large numbers on $\KK$
(Theorem~\ref{t:LLN-intro}) by collecting various results we have
already proved.

\begin{proof}[Proof of Theorem~\ref{t:LLN-intro}]
The fully sticky case is immediate from Theorem~\ref{t:stickyLLN}.
Consider the partly sticky case.  By Corollary~\ref{cor:Gammamax}
applied to the empirical measure~$\mu_N$, the empirical barycenter is
$b_N = \0$ or $b_N = (m^N_{\theta^*,1}, \theta^*)$ where $\theta^*$ maximizes
$\theta \mapsto m^N_{\theta,1}$.  Combining this fact with
Lemma~\ref{l:mconv} leads to the conclusion that
$$%
  \limsup_{N \to \infty} d(b_N,\0)
  =
  \limsup_{N \to \infty} m^N_{\theta^*,1}
  \leq
  \max_\theta m_{1,\theta}
$$
holds $\Pm$-almost surely.  In the partly sticky case,
$m_{1,\theta}\leq 0$ for all $\theta$.  Thus $b_N \to \0$ holds
$\Pm$-almost surely.  The other statements in the partly sticky case
follow from Theorem~\ref{t:stickyLLN}.

Finally, consider the nonsticky case.  Convergence $b_N \to \bar b$
again follows from the representation $b_N = (m^N_{\theta^*,1},
\theta^*)$ where $\theta^*$ maximizes $\theta \mapsto m^N_{\theta,1}$.
By Lemma~\ref{l:mconv} $\Pm$-almost surely any maximizer $\theta^N$ of
$\theta \mapsto m^N_{\theta,1}$ converges, as $N \to \infty$, to the
maximizer of $\theta \mapsto m_{\theta,1}$, which is unique in the
nonsticky case.  By definition of $\bar b$, this implies that $b_N \to
\bar b$, $\Pm$-almost~surely.
\end{proof}

\section{Proofs of the central limit theorems}\label{s:clt1}

This section contains proofs of the three central limit theorems:
Theorem~\ref{t:fullstickyclt}, Theorem~\ref{t:stickyclt}, and
Theorem~\ref{t:non-stickyclt}.  First comes the fully sticky case,
which follows almost immediately from Theorem~\ref{t:LLN-intro}.

\begin{proof}[Proof of Theorem~\ref{t:fullstickyclt}]
Let $N^*$ be the random integer from Theorem~\ref{t:LLN-intro},
which has the property that, $\Pm$-almost surely, $b_N =0$ for all
$N \geq N^*(\omega)$.  If $\phi: \KK \rightarrow \Rm$ is any bounded
function then
\begin{align*}
\Big| \int\phi(p)\,d\nu_N(p) - \phi(\0)\Big|
   & = \big|\Em\phi(b_N) - \phi(\0)\big|
\\ & = \big|\Em(\phi(b_N) - \phi(\0))\mathbf{1}_{N < N^*}\big|
     \leq 2\big(\sup_{p \in \KK} |\phi(p)|\big) \Pm(N < N^*).
\end{align*}
Since $N^*$ is almost surely finite, $\Pm(N < N^*) \rightarrow 0$ as
$N \rightarrow \infty$ which concludes the proof.  Since the bound on
the right hand side depends only on the supremum norm of~$\phi$, the
bound also implies convergence in the total variation norm.
\end{proof}

Next comes the proof of the central limit theorem in the partly
sticky case.

\begin{proof}[Proof of Theorem~\ref{t:stickyclt}]
Let $K = [A,B]$ be the interval on which $m_{\theta,1} = 0$, so
$m_{\theta,1} < 0$ for all $\theta \notin [A,B]$ by hypothesis.  Recall that $\theta^*$ is the midpoint of this interval. Let
$\epsilon \in (0,\pi/4)$.  By Theorem~\ref{t:stickyLLN} there is an
integer $N^*(\omega,\epsilon)$ such that, almost surely, $b_N(\omega)
\in C_{[A_\epsilon, B^\epsilon]}$ if $N \geq N^*(\omega,\epsilon)$,
where $A_\epsilon = A - \epsilon$ and $B^\epsilon = B + \epsilon$.
Since $\nu_N$ is the distribution of the random variable $\sqrt{N}
b_N$ on~$\KK$,
$$%
  \lim_{N \to \infty} \nu_N( C_{[A_\epsilon, B^\epsilon]}) = 1.
$$
Therefore
\begin{eqnarray}
&& \lim_{N\to \infty} \left(\int_{\mathcal{K}} \phi_1(p)\,d\nu_N(p) - \int_{\mathcal{K}} \phi_2(p)\,d\nu_N(p)\right)
 \no \\
&& \quad \quad \quad = \lim_{N\to \infty} \left(\int_{C_{[A_\epsilon,B^\epsilon]}} \phi_1(p)\,d\nu_N(p) - \int_{C_{[A_\epsilon,B^\epsilon]}}\phi_2(p)\,d\nu_N(p)\right) \no
\end{eqnarray}
holds for any bounded continuous function $\phi_1, \phi_2:\mathcal{K} \to \mathbb R$. For this reason it suffices to prove (\ref{stickyclt1}) for continuous bounded functions differing only on $C_{[A_\epsilon,B^\epsilon]}$. Such functions are of the form $\phi = \varphi
\circ F_{\theta^*}$ where $\varphi:\RR^2 \to \RR$ is continuous and
bounded.

\excise{
Therefore
$$%
  \lim_{N \to \infty} \left| \int_\KK \phi_1(p) \,d \nu_N(p) -
  \int_\KK \phi_2(p) \,d \nu_N(p) \right| = 0
$$
holds for bounded functions $\phi_1$ and~$\phi_2$ that agree on
$C_{[A_\epsilon, B^\epsilon]}$.  For any continuous bounded function
$\phi_1: \KK \to \RR$, there is a continuous bounded function
$\varphi: \RR^2 \to \RR$ such that the composite $\phi_2 = \varphi
\circ F_{\theta^*}: \KK \to \RR$ satisfies $\phi_1(p) = \phi_2(p)$ for
all $p \in C_{[A_\epsilon, B^\epsilon]}$.  Therefore, it suffices to
prove (\ref{stickyclt1}) for functions of the form}


Using the convex projection $\hat P_\rho$ from~(\ref{eq:convex
projection}) for $\rho = \frac 12 |A - B|$, let $\zeta_N$ denote the
measure on~$\RR^2$ defined by $\Pm(\sqrt{N} \hat
P_\rho(\eta_{\theta^*,N}) \in W) = \zeta_N(W)$ for Borel sets $W
\subset \RR^2$.  Then $\expE[\eta_{\theta^*,N}] = 0$, because
$m_\theta = \expE[\eta_{\theta,N}]$ for all $\theta\in \RR/\alpha \ZZ$
and $m_{\theta^*} = 0$ by hypothesis.  Recalling
Remark~\ref{r:squareint}, which guarantees square-integrability, the
standard CLT for $\eta_{\theta^*,N}$ in $\RR^2$ implies that the law
of $\sqrt{N} \eta_{\theta^*,N}$ converges to $g$, the law of the
multivariate normal with covariance~(\ref{eq:g}).  Thus
\begin{equation}\label{zetanormal}
  \lim_{N \to \infty} \int_{\RR^2} \varphi(z) d\zeta_N(z) =
  \int_{\RR^2} \varphi(z) d (g \circ \hat P_\rho^{-1}(z))
\end{equation}
holds for any continuous bounded function $\varphi:\RR^2 \to \RR$.  We
claim that for any $\delta > 0$ there is an integer $N_\delta$ such
that
\begin{equation}\label{bNetaNclose}
  \Pm( \sqrt{N}\| F_{\theta^*}  b_N - \hat P_\rho \eta_{{\theta^*},N}
  \| > \delta)  \leq \delta
\end{equation}
holds for all $N \geq N_\delta$.  This estimate and~(\ref{zetanormal})
imply that
\begin{equation}\label{tildenuconv}
  \lim_{N \to \infty} \int_{\RR^2} \varphi(z) d \tilde \nu_N(z) =
  \int_{\RR^2} \varphi(z) d (g \circ \hat P_\rho^{-1}(z))
\end{equation}
where $\tilde \nu_N = \nu_N \circ F^{-1}_{\theta^*}$ is the law of
$\sqrt{N} F_{\theta^*} b_N$ on $\RR^2$.

Recall that $F_{\theta^*}: C_{[A_\epsilon, B^\epsilon]}\to
D_{\rho+\epsilon}$ is bijective, where the sector $D_{\rho+\epsilon}
\subset \RR^2$ is defined by replacing $\rho$ with $\rho+\epsilon$
in~(\ref{eq:Drho}), and $\nu_N( C_{[A_\epsilon, B^\epsilon]}) \to 1$
as $N \to \infty$.  Combining this with~(\ref{tildenuconv}) leads to
the conclusion that~(\ref{stickyclt1}) holds for the continuous
bounded function $\phi = \varphi \circ F_{\theta^*}$:
\begin{align*}
\lim_{N \to \infty} \int_\KK \varphi(F_{\theta^*}(p)) d \nu_N(p)
  & = \lim_{N \to \infty} \int_{\RR^2} \varphi(z) d \tilde \nu_N(z)\\
  & = \int_{\RR^2} \varphi(z) d (g \circ \hat P_\rho^{-1}(z))\\
  & = \int_\KK \varphi(F_{\theta^*}(p)) d (g \circ \hat P_\rho^{-1}
      \circ F_{\theta^*}(p)).
\end{align*}
It remains to prove~(\ref{bNetaNclose}) by estimating $\|F_{\theta^*}
b_N - \hat P_\rho \eta_{{\theta^*},N}\|$.

First, suppose $b_N = (r,\hat\theta) \in C^+_{[A,B]}$.  If $A=B$ then
$\hat\theta = \theta^*$ and thus $\eta_{{\theta^*},N} = F_{\theta^*}
b_N = \hat P_\rho F_{\theta^*} b_N$ by Lemma~\ref{lem:bNetaN2}.  Now
assume $A \neq B$.  Then $\mu(\Ical_\theta) = 0$ for all $\theta \in
[\theta^*,\hat\theta]$ by Corollary~\ref{cor:noshadowmass}, as by
hypothesis $m_{\theta}=0$ for all $\theta \in [\theta^*,\hat\theta]$
and $|\hat\theta-\theta^*|\leq |B-A| < \pi$.  This implies that also
$\nu_N(\Ical_{\theta}) = 0$ for all $\theta \in
[\theta^*,\hat\theta]$.  Since $r > 0$, Lemma~\ref{lem:bNetaN2}
implies that $\eta_{{\theta^*},N} = F_{\theta^*} b_N = \hat P_\rho
F_{\theta^*} b_N$, as desired.

For the remainder of the proof, let $\epsilon \in (0,\pi/4)$ and
assume $b_N \in C_{[A_\epsilon,B^\epsilon]}$ but $b_N \notin
C^+_{[B,A]}$.  Suppose $b_N = (r,\hat\theta)$ with $\hat\theta \in
[B,B^\epsilon]$ and $r \geq 0$; the case $\hat\theta \in
[A_\epsilon,A]$ is treated in the same way.  By
Corollary~\ref{cor:Gammamax} and Lemma~\ref{lem:Dmtheta},
$m^N_{\hat\theta,1} = r$ and $m^N_{\hat\theta,2} = 0$.  Denote by
$$%
  w_N(s) = \frac{1}{N} \sum_{p_n \in \Ical_s} d(\0,p_n),
$$
the sample analog of $w(s)$ from~(\ref{wpmdef2}).  Utilizing the
second equation in Corollary~\ref{cor:odeintegrate},
$$%
  m^N_{B,2} = m^N_{\hat\theta,1} \sin(\hat\theta-B) +
  \int_{B}^{\hat\theta}w_N(\psi) \cos(\psi-B)\,d\psi,
$$
which implies that $m^N_{B,2} \geq 0$.  Moreover, by the first equation of
Corollary~\ref{cor:odeintegrate},
$$%
  m^N_{B,1} = m_{\hat\theta,1}^N \cos(\hat\theta - B) -
  \int_{B}^{\hat\theta} w_N(\psi) \sin(\psi - B)\,d\psi.
$$
Therefore $m^N_{\hat\theta,1} \geq m^N_{B,1} \geq 0$.  Similarly also
\begin{align*}
m^N_{\hat\theta,1}
   & = m_{B,1}^N \cos(\hat\theta - B) + m_{B,2}^N \sin(\hat\theta - B)
       - \int_B^{\hat\theta} w(\psi) \sin(\hat\theta - \psi)\,d\psi
\\ & \leq m_{B,1}^N + \epsilon m_{B,2}^N.
\end{align*}
This shows that $r \in [m_{B,1}^N, m_{B,1}^N + \epsilon m_{B,2}^N]$.
For later use, note that for $\epsilon > 0$ sufficiently small,
\begin{equation}\label{m2B:ineq}
  m^N_{B,2}
  \,\leq\,
  (m^N_{\hat\theta,1} + \bar r)\epsilon
  \,\leq\,
  (m^N_{B,1}+\bar r)\epsilon + m^N_{B,2}\epsilon^2
  \,\leq\,
  3\bar r \epsilon.
\end{equation}
Observe that $\Phi_\rho m^N_B = m^N_{\theta^*}$.  If $A = B$ this is
obvious because $\rho = 0$ and $\theta^* = B$.  If $A \neq B$, then
this follows from Lemma~\ref{lem:mthetamodulus} because
$\nu^N(\Ical_\theta) = 0$ for all $\theta \in [\theta^*,B]$, due to
$\mu(I_\theta) = 0$.  Therefore $\hat P_\rho m^N_{\theta^*} = \hat
P_\rho \Phi_\rho m^N_B = (m_{B,1}^N,\rho)$ in polar coordinates,
because convex projection commutes with rotation,
cf.~Figure~\ref{f:end-of-clt-proof}.  In conjunction with
$F_{\theta^*} b_N = (r, \hat\theta -\theta^*)$, therefore
\begin{figure}[t!]
\centering
\begin{minipage}{0.2\linewidth}
\begin{excise}{
  \begin{tikzpicture}[scale=3]
  \coordinate (O) at (0,0);

  \coordinate (a) at (60:.5);
  \coordinate (A) at (30:1);
  \coordinate (B) at (0:1);
  \coordinate (C) at (-30:1);
  \coordinate (aAB) at ($(A)!(a)!(O)$);

  \draw[dashed] (a)--(aAB);
  \fill (a) circle[radius=.5pt] node[above]{$m_{\theta^*}^N$};
  \fill (aAB) circle[radius=.5pt];

  \draw[solid] (A)--(O)--(C);
  \draw[dashed] (O)--(B);

  \draw pic[draw=orange,arrows={latex[orange, fill=orange]}-{latex[orange,
        fill=orange]},solid,thick, angle eccentricity=1.05, angle
	radius=2cm,right,"$\rho$"] {angle=B--O--A};
  \draw pic[draw=orange,arrows={latex[orange,
        fill=orange]}-{latex[orange, fill=orange]},solid,thick, angle
        eccentricity=1.05, angle radius=2cm,right,"$\rho$"] {angle=C--O--B}; 
  \draw[opacity=0,white] (0,-.95) rectangle (1,.75);
  \end{tikzpicture} 
}\end{excise}
\includegraphics[width=1\textwidth]{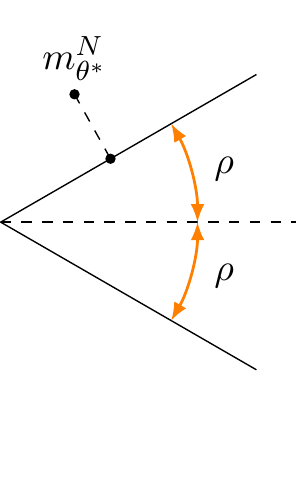}
\end{minipage}
\begin{minipage}{0.2\linewidth}
$$%
  \xleftarrow{\quad\Phi_\rho \quad}
$$
\end{minipage}
\begin{minipage}{0.2\linewidth}
\begin{excise}{
  \begin{tikzpicture}[scale=3]
  \coordinate (O) at (0,0);

   \coordinate (A) at (0:1);
   \coordinate (B) at (-60:1);

  \coordinate (a) at (30:.5);
  \coordinate (aAB) at ($(A)!(a)!(O)$);

  \draw[solid] (A)--(O)--(B);
  \draw pic[draw=orange,arrows={latex[orange, fill=orange]}-{latex[orange,
        fill=orange]},solid,thick, angle eccentricity=1.05, angle
	radius=2.45cm,right,"$2\rho$"] {angle=B--O--A};

  \draw[dashed] (a)--(aAB);
  \fill (a) circle[radius=.5pt] node[above]{$m_{B}^N$};
  \fill (aAB) circle[radius=.5pt] node[below ]{$(m_{B,1}^N,0)$};
  \draw[opacity=0,white] (0,-.95) rectangle (1,.75);
  \end{tikzpicture} 
}\end{excise}
\includegraphics[width=1\textwidth]{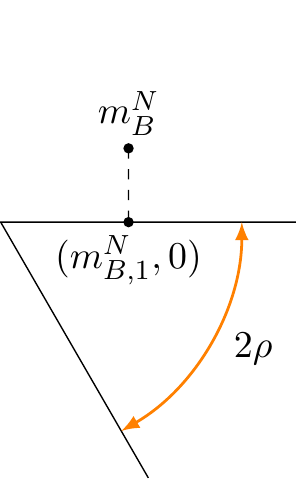}
\end{minipage}
\caption{\label{f:end-of-clt-proof}%
Detail for the proof of (\ref{bNetaNclose}): Convex projection
commutes with rotation.}
\end{figure}
\begin{align*}
\|F_{\theta^*} b_N - \hat P_\rho m_{\theta^*}^N\|^2
   &=\big(r\cos(\hat\theta - \theta^*) - m_{B,1}^N \cos\rho\big)^2 +
     \big(r\sin(\hat\theta - \theta^*) - m_{B,1}^N \sin\rho\big)^2 
\\ &=r^2+ (m_{B,1}^N)^2 - 2r m_{B,1}^N \cos(\hat\theta-B)
\\ &= (r - m_{B,1}^N)^2 + 2r m_{B,1}^N \big(1-\cos(\hat\theta-B)\big)
\\ &\leq (\epsilon m_{B,2}^N)^2 + (m_{B,1}^N + \epsilon m_{B,2}^N) m_{B,1}^N \epsilon^2
\\ &= \epsilon^2(m_{B,1}^N + m_{B,2}^N)^2 -\epsilon^2 m_{B,1}^N m_{B,2}^N (2-\epsilon).
\end{align*}
By applying the same argument when $\hat\theta \in [A_\epsilon,A]$,
upon noting that $m^N_{A,1}, m^N_{A,2} \leq 0$, we conclude that for
$\epsilon$ sufficiently small and $b_N \in C_{[A_\epsilon,A)\cup(B,
B^\epsilon]}$,
$$%
  \|F_{\theta^*} b_N - \hat P_\rho m_{\theta^*}^N\| \leq \epsilon
  \left(m^N_{B,1}+m^N_{B,2}-m^N_{A,1}-m^N_{A,2} \right).
$$
Let $X_N = m^N_{B,1}+m^N_{B,2}-m^N_{A,1}-m^N_{A,2}$; each term in this
sum is the average of $N$ independent random variables in $\RR^2$, and
each term has zero mean since $\mathbb E(m_A^N) = m_A = 0$ and
$\mathbb E(m_B^N)= m_B=0$, by hypothesis.  The Chebychev inequality
implies
\begin{align}
  \Pm\Big(b_N\in C_{[A_\epsilon,B^\epsilon]}, \sqrt{N}\big\|F_{\theta^*}
  b_N - \hat P_\rho \eta_{{\theta^*},N} \big\| > \delta\Big)
  \nonumber&
  \leq \Pm\Big(\sqrt{N}\epsilon |X_N| > \delta\Big)
\\\nonumber&
  \leq \frac{\epsilon^2\mathbb E\big(X^2_N\big)N}{\delta^2}
\\\label{eq:markov}&
  \leq\frac{\delta}{2}\text{ for }\epsilon = \sqrt{C\delta^3/2}
\end{align}
by square-integrability with a constant $C$ that depends only on $\mu$.

By Theorem~\ref{t:stickyLLN} there is an integer
$N^*(\omega,\epsilon)$ such that $b_N \in C_{[B_\epsilon,
A^\epsilon]}$ if $N \geq N^*(\omega,\epsilon)$ for almost surely all
$\omega$.  In particular, given $\delta >0$ there is an integer
$N_{\epsilon,\delta}$ such that
$$%
  \Pm( b_N \in C_{[B_\epsilon, A^\epsilon]} ) \geq 1 - \delta/2
  \text{ for all } N \geq N_{\epsilon,\delta}.
$$
Setting $N_\delta = N_{\epsilon,\delta}$ for $\epsilon=
\sqrt{C\delta^3/2}$ with (\ref{eq:markov}), the above yields the
desired claim~(\ref{bNetaNclose}).
\end{proof}

We conclude with the proof of the central limit theorem for the
nonsticky case.

\begin{proof}[Proof of Theorem~\ref{t:non-stickyclt}]
In the nonsticky case, the barycenter of $\mu$, denoted $\bar b$, is
equal to $(r^*,\theta^*) \in \KK$ where $r^*=m_{\theta^*,1} > 0$ and
$\theta^*$ is the unique angle that maximizes $\theta \mapsto
m_{\theta,1}$.  By Theorem~\ref{t:llnAint}, $b_N \in
C^+_{[\theta^*-\epsilon,\theta^*+\epsilon]}$ for all $N$ sufficiently
large, given any fixed $\epsilon \in (0, \pi/2)$.

The standard CLT for $m_{\theta^*}^N$ in $\Rm^2$ implies that the law of $\sqrt{N}
(m_{\theta^*}^N- F_{\theta^*} \bar b)$ converges weakly to $g$. In cartesian coordinates, $F_{\theta^*} \bar b = (e_1 \cdot F_{\theta^*} \bar b, e_2 \cdot F_{\theta^*} \bar b) = (r^*,0)$. Therefore, to show that the law of the random vector $\sqrt{N}\big(e_1 \cdot F_{\theta^*} b_N - r^*, (1 + \kappa) e_2 \cdot F_{\theta^*} b_N\big)$ also converges weakly to $g$ as $N \to \infty$ (the random variable $\kappa \geq 0$ was defined at (\ref{eq:kappa})), it suffices to show that
for any
$\delta > 0$,
\begin{equation}
\lim_{N \to \infty}  \Pm\Big( \sqrt{N} \norm{\big(e_1 \cdot F_{\theta^*} b_N, (1 + \kappa) e_2
  \cdot F_{\theta^*} b_N\big) - m_{\theta^*}^N } \geq \delta \Big)
    = 0.  \label{goal1}
\end{equation}

Recall from (\ref{eq:Gammaemp}) that the
empirical mean $b_N$ is the unique minimizer of
$$%
  (r,\theta) \mapsto \Gamma(r,\theta) = \frac{r^2}{2} - r
  m_{\theta,1}^N + \Gamma_N(0).
$$
That is, if $b_N = (\hat r, \hat\theta)$, then $\hat\theta$ is the
unique maximizer of the function
$$%
  \theta \mapsto f(\theta) = \frac{(m_{\theta,1}^N)^2}{2}.
$$
The first objective is to show that $|\hat \theta - \theta^*| = O(1/\sqrt{N})$, meaning
that for any $\epsilon > 0$ there are constants $N_\epsilon$,
$C_\epsilon > 0$ such that $\Pm\left( \sqrt{N} |\hat \theta - \theta^*| >
  C_\epsilon \right) \leq \epsilon$ for all $N \geq N_\epsilon$.  Using Corollary~\ref{cor:odeintegrate}, write $m_{\theta,1}^N$ in
terms of $\theta^*$:
\begin{equation}\label{thetarotate}
  m_{\theta,1}^N = m_{\theta^*,1}^N \cos(\theta - \theta^*) +
  m_{\theta^*,2}^N \sin(\theta - \theta^*) - \int_{\theta^*}^\theta
  w_N(\psi) \sin(\theta - \psi) \,d\psi, 
\end{equation}
where
$$%
  w_N(\psi) = \frac{1}{N} \sum_{p_n \in \Ical_\psi} d(\0,p_n).
$$
Because $m_{\theta^*}^N$ satisfies the central limit theorem and because $m_{\theta^*,2} = 0$, this implies that
\begin{align}
  m_{\theta,1}^N
  \no
 & =
    m_{\theta^*,1} \cos(\theta - \theta^*) + O(1/\sqrt{N})
   \cos(\theta - \theta^*) 
\\ \label{mthetaexpasion1}
&  \phantom{=\ }
   \mbox{} + O(1/\sqrt{N})\sin(\theta - \theta^*) - \int_{\theta^*}^\theta
  w_N(\psi) \sin(\theta - \psi) \,d\psi.
\end{align}
For $|\theta - \theta^*| < \pi/2$ the function
\[
\theta \mapsto \int_{\theta^*}^\theta
- w_N(\psi) \sin(\theta - \psi) \,d\psi \leq 0
\]
has a maximum at $\theta = \theta^*$. In view of this and (\ref{mthetaexpasion1}), we conclude that the angle $\hat \theta$ at which the maximum in $\theta \mapsto m^N_{\theta,1}$ is attained must satisfy $|\hat \theta - \theta^*| \leq O(1/\sqrt{N})$.

Now we compare $F_{\theta^*} b_N$ to $m_{\theta^*}^N$ to derive (\ref{goal1}).  Recall that $\Phi_\sigma:\Rm^2 \to \Rm^2$ denotes rotation by angle $\sigma$  (defined just before Lemma \ref{lem:mthetamodulus}). When $|\hat \theta - \theta^*| < \pi$ (which happens almost surely as $N \to \infty$) we have $F_{\theta^*} b_N = \Phi_{\hat \theta - \theta^*} F_{\hat \theta} b_N$. Therefore, by Lemma \ref{lem:bNetaN2}, we have
\begin{equation}
F_{\theta^*} b_N = \Phi_{\hat \theta - \theta^*} F_{\hat \theta} b_N = \Phi_{\hat \theta - \theta^*}  m_{\hat \theta}^N \label{Fhatstar}
\end{equation}
for $N$ large enough. By Corollary \ref{cor:odeintegrate} we also have
\begin{equation}
m_{\hat \theta}^N = \Phi_{\hat \theta - \theta^*}^{-1} m_{\theta^*}^N - V \label{mhatstar}
\end{equation}
where $V = (V_{1,N},V_{2,N})$ is the vector with components
\[
V_{1,N} = \int_{\theta^*}^{\hat \theta} w_N(\psi) \sin(\hat \theta - \psi) \,d \psi, \quad \quad V_{2,N} = \int_{\theta^*}^{\hat \theta} w_N(\psi) \cos(\hat \theta - \psi) \,d \psi.
\]
Hence 
\begin{equation}
e_1 \cdot F_{\theta^*} b_N  - e_1 \cdot m_{\theta^*}^N  = - e_1 \cdot \Phi_{\hat \theta - \theta^*} V = - \cos(\hat \theta - \theta^*) V_{1,N} + \sin(\hat \theta - \theta^*) V_{2,N} \label{e1Ftheta}
\end{equation}
and
\begin{equation}
e_2 \cdot F_{\theta^*} b_N  - e_2 \cdot m_{\theta^*}^N  = - e_2 \cdot \Phi_{\hat \theta - \theta^*} V =  \sin(\hat \theta - \theta^*) V_{1,N} - \cos(\hat \theta - \theta^*) V_{2,N}\label{e2Ftheta}
\end{equation}
for $N$ sufficiently large. Using the fact that $|\hat \theta - \theta^*| \leq O(1/\sqrt{N})$, we find that $|V_{1,N}| = O(1/N)$ and $|V_{2,N}| = O(1/\sqrt{N})$: indeed, 
\[
0 \leq \inf_\psi w_N(\psi) \leq \sup_\psi  w_N(\psi) \leq \frac{1}{N} \sum_{n=1}^N d(\0,p_n)
\]
and the latter converges to $\bar r < \infty$ (recall (\ref{integ})) almost surely as $N \to \infty$. Hence, with probability one,
\[
V_{1,N} \leq 2 \bar r \int_{\theta^*}^{\hat \theta} \sin(\hat \theta - \psi) \,d \psi \leq \bar r (\hat \theta - \theta^*)^2 , \quad \text{and} \quad \quad |V_{2,N}| \leq 2 \bar r|\hat \theta - \theta^*|
\]
hold for all $N$ large enough. Applying this at (\ref{e1Ftheta}) and using $|\hat \theta - \theta^*| \leq O(1/\sqrt{N})$, we obtain
\[
e_1 \cdot F_{\theta^*} b_N  - e_1 \cdot m_{\theta^*}^N  = O(1/N),
\]
by which we mean that for any $\epsilon > 0$, there is $C_\epsilon$ such that
\begin{equation}
\limsup_{N \to \infty} \Pm \left( |e_1 \cdot F_{\theta^*} b_N  - e_1 \cdot m_{\theta^*}^N | \geq C_\epsilon/N \right) \leq \epsilon. \label{goal3}
\end{equation}
In particular, $e_1 \cdot F_{\theta^*} b_N  - e_1 \cdot m_{\theta^*}^N$ is $o(1/\sqrt{N})$ in the sense of (\ref{goal1}).

To complete the proof of (\ref{goal1}), we must show that $(1 + \kappa) e_2 \cdot F_{\theta^*} b_N  - e_2 \cdot m_{\theta^*}^N$ is $o(1/\sqrt{N})$, as well. We will use (\ref{e2Ftheta}) and a more subtle estimate of $V_{2,N}$ and of $\hat \theta - \theta^*$.  From (\ref{mhatstar}) and the fact that $e_2 \cdot m_{\hat \theta}^N = m_{\hat \theta,2}^N = 0$ (by Lemma~\ref{lem:bNetaN2}), we have 
\begin{align}
  0 = m_{\hat \theta,2}^N &   =
  -m_{\theta^*,1}^N \sin(\hat \theta - \theta^*) + m_{\theta^*,2}^N
  \cos(\hat \theta - \theta^*) - V_{2,N} \no \\
& = - m^N_{\theta^*,1} (\hat\theta - \theta^*) + m_{\theta^*,2}^N  + O(N^{-1}) - V_{2,N}. \label{m2wN}
\end{align}
(We used $|\hat \theta - \theta^*| \leq O(1/\sqrt{N})$ again in the last equality.) As the next lemma shows, the integral term $V_{2,N}$ is approximated by $r^* \kappa (\hat \theta - \theta^*)$, where the random variable $\kappa$ was defined at (\ref{eq:kappa}): $r^* \kappa = w^-(\theta^*)$ if $\hat \theta > \theta^*$, and $r^* \kappa = w^+(\theta^*)$ if $\hat \theta < \theta^*$, and $r^* \kappa = 0$ if $\hat \theta = \theta^*$.

\begin{lem} \label{lem:V2}
Let $\hat \theta$ be the angular coordinate of $b_N$. Let $U^+_N$ be the event that $\hat \theta > \theta^*$, and let $U^-_N$ be the event that $\hat \theta < \theta^*$. If $Z_N$ is the random variable
\begin{align}
Z_N & =  \mathbb{I}_{U^+_N} \cdot \left| \int_{\theta^*}^{\hat \theta}
    w_N(\psi) \cos(\hat\theta - \psi) \,d\psi - w^-(\theta^*)(\hat \theta - \theta^*)  \right| \no \\
& \phantom{=\ } + \mathbb{I}_{U^-_N} \cdot \left| \int_{\theta^*}^{\hat \theta}
    w_N(\psi) \cos(\hat\theta - \psi) \,d\psi - w^+(\theta^*)(\hat \theta - \theta^*)  \right|,
\end{align}
then $Z_N$ is $o(1/\sqrt{N})$ in probability as $N \to \infty$: for any $\delta > 0$, 
\[
\lim_{N \to \infty} \Pm( Z_N > \delta/\sqrt{N}) = 0.
\]
\end{lem}

By combining Lemma \ref{lem:V2} and (\ref{m2wN}), we derive
\begin{align}
0 & = - m^N_{\theta^*,1} (\hat\theta - \theta^*)  +  m_{\theta^*,2}^N - r^* \kappa (\hat \theta - \theta^*) + O(1/N) + o(1/\sqrt{N}),
\end{align}
and thus
\begin{equation}
  \hat\theta - \theta^* = \frac{m_{\theta^*,2}^N}{m_{\theta^*,1}^N +
  r^* \kappa } + o(1/\sqrt{N}). \label{thetadiff}
\end{equation}
Recalling that $|V_{1,N}| = O((\hat \theta - \theta^*)^2) = O(1/N)$, we now combine (\ref{e2Ftheta}) with Lemma \ref{lem:V2} and (\ref{thetadiff}) to obtain
\begin{align}
  e_2 \cdot F_{\theta^*} b_N - m_{\theta^*,2}^N  & = - V_{2,N} + O(1/N) \no \\
& = - r^* \kappa (\hat \theta - \theta^*) +  o(1/\sqrt{N}) \no \\
& =  - m_{\theta^*,2}^N \frac{r^* \kappa }{m_{\theta^*,1}^N +
   r^* \kappa } + o(1/\sqrt{N}). \label{goal2}
\end{align}
In the case $w^\pm(\theta^*) = 0$, we have $\kappa = 0$, so (\ref{goal1}) follows from (\ref{goal2}) and (\ref{goal3}). However, when $w^\pm(\theta^*) \neq 0$, (\ref{goal2}) implies that
\begin{align}
  e_2 \cdot F_{\theta^*} b_N - m_{\theta^*,2}^N  & =  - m_{\theta^*,2}^N \frac{r^* \kappa }{r^* +
   r^* \kappa } + o(1/\sqrt{N}), \no
\end{align}
because $m_{\theta^*,1}^N \to r^*$ and $m_{\theta^*,2}^N \to m_{\theta^*,2} = 0$ as $N \to \infty$. Therefore,
$$%
  (1 + \kappa) e_2 \cdot F_{\theta^*} b_N -  m_{\theta^*,2}^N =  o(1/\sqrt{N}).
$$
This and (\ref{goal3}) imply (\ref{goal1}), as desired. Except for the proof of Lemma \ref{lem:V2}, the proof of Theorem \ref{t:non-stickyclt} is complete.
\end{proof}

\begin{proof}[Proof of Lemma \ref{lem:V2}]
We will restrict our attention to the case that $\hat \theta > \theta^*$ (in the event $U_N^+$, which is equivalent to $e_2 \cdot F_{\theta^*} b_N > 0$); the other case is analyzed in the same way.  We decompose the integral as
\begin{eqnarray}
\int_{\theta^*}^{\hat \theta} w_N(\psi) \cos(\hat \theta - \psi) \,d\psi & = & w^-(\theta^*) \int_{\theta^*}^{\hat \theta}  \cos(\hat \theta - \psi) \,d\psi  +  \left(w_N^-(\theta^*) - w^-(\theta^*) \right) \int_{\theta^*}^{\hat \theta}  \cos(\hat \theta - \psi) \,d\psi \no \\
& & + \int_{\theta^*}^{\hat \theta}  \left(w_N(\psi) - w_N^-(\theta^*) \right)  \cos(\hat \theta - \psi) \,d\psi \no \\
& = & T_1 + T_2 + T_3, \label{T123def}
\end{eqnarray}
where 
\[
w_N^-(\psi) = \frac{1}{N} \sum_{p_n \in \Ical_{\psi}^- } d(\0,p_n).
\]
Now we estimate each of the terms $T_1$, $T_2$, and $T_3$, in (\ref{T123def}) using the fact that $\hat \theta - \theta^* = O(1/\sqrt{N})$, which was proved independently of this lemma. First,
\begin{eqnarray}
T_1 & = & w^-(\theta^*) \int_{\theta^*}^{\hat \theta}  \cos(\hat \theta - \psi) \,d\psi \no \\
& = &  w^-(\theta^*) \left( (\hat \theta - \theta^*) + O((\hat \theta - \theta^*)^3) \right) = w^-(\theta^*)(\hat \theta - \theta^*) + O(1/N). \no
\end{eqnarray}
For $T_2$, we apply the CLT to 
\[
\left(w_N^-(\theta^*) - w^-(\theta^*) \right) = \frac{1}{N} \sum_{n=1}^N \left( d(\0,p_n) \mathbb{I}_{\Ical^-_{\theta^*}}(p_n) - \expE[ d(\0,p_n) \mathbb{I}_{\Ical^-_{\theta^*}}(p_n)] \right)
\]
which is a sum of independent, identically-distributed random variables with zero mean and finite variance (due to square integrability condition (\ref{squareint})). Hence $\text{Var}\left(w_N^-(\theta^*) - w^-(\theta^*) \right) = O(1/N)$ and $\left(w_N^-(\theta^*) - w^-(\theta^*) \right) = O(1/\sqrt{N})$, which implies that
\[
T_2 = \left(w_N^-(\theta^*) - w^-(\theta^*) \right) \int_{\theta^*}^{\hat \theta}  \cos(\hat \theta - \psi) \,d\psi = O(1/\sqrt{N}) O(\hat \theta - \theta^*) = O(1/N).
\]

Finally, we show that the term $T_3 =  \int_{\theta^*}^{\hat \theta}  \left(w_N(\psi) - w_N^-(\theta^*) \right)  \cos(\hat \theta - \psi) \,d\psi$ is $o(1/\sqrt{N})$. If $\Ical_{\psi} \Delta \Ical^-_{\theta^*} = (\Ical_{\psi} \cup \Ical^-_{\theta^*} ) \setminus (\Ical_{\psi} \cap \Ical^-_{\theta^*})$ denotes the symmetric difference of the shadows, then 
\[
\left|w_N(\psi) - w_N^-(\theta^*) \right| \leq \frac{1}{N} \sum_{n=1}^N \left( d(\0,p_n) \mathbb{I}_{\Ical_{\psi} \Delta \Ical^-_{\theta^*}}(p_n) \right).
\]
Recall that we are assuming $\theta^* \leq \hat \theta$.  We may also assume that $|\hat \theta - \theta^*| < \min(\alpha - \pi, \pi/2)$ (which happens with probability approaching $1$ as $N \to \infty$), then $\Ical_{\psi} \Delta \Ical^-_{\theta^*}  \subset \Ical_{\hat \theta} \Delta \Ical^-_{\theta^*}$ for $\psi \in (\theta^*, \hat \theta)$, and therefore
\begin{eqnarray}
|T_3| & \leq & \int_{\theta^*}^{\hat \theta}  \left|w_N(\psi) - w_N^-(\theta^*) \right|  \cos(\hat \theta - \psi) \,d\psi \no \\
& \leq  & \frac{1}{N} \sum_{n=1}^N  d(\0,p_n)   \mathbb{I}_{\Ical_{\hat \theta} \Delta \Ical^-_{\theta^*}}(p_n) \int_{\theta^*}^{\hat \theta} \cos(\hat \theta - \psi) \,d\psi \no \\
& \leq  & \frac{1}{N} \sum_{n=1}^N  d(\0,p_n)   \mathbb{I}_{\Ical_{\hat \theta} \Delta \Ical^-_{\theta^*}}(p_n) |\hat \theta - \theta^*|.
\end{eqnarray}
Fix $\epsilon > 0$ small and let $C_\epsilon > 0$ be such that $\Pm( \hat \theta \in [\theta^* -  \frac{C_\epsilon}{\sqrt{N}}, \theta^* + \frac{C_\epsilon}{\sqrt{N}}]) \geq 1 - \epsilon$ for all $N$ large enough. Then, with probability exceeding $1 - 2\epsilon$, we have
\[
\mathbb{I}_{U_N^+} \cdot |T_3| \leq \frac{C_\epsilon}{\sqrt{N}} \frac{1}{N} \sum_{n=1}^N  d(\0,p_n)   \mathbb{I}_{\Ical_{\hat \theta} \Delta \Ical^-_{\theta^*}}(p_n)\leq  \frac{C_\epsilon}{\sqrt{N}} \frac{1}{N} \sum_{n=1}^N  d(\0,p_n)   \mathbb{I}_{\Ical_{\bar \theta_N^\epsilon} \Delta \Ical^-_{\theta^*}}(p_n)
\]
for $N$ large enough, where the angle $\bar \theta_N^\epsilon = \theta^* + \frac{C_\epsilon}{\sqrt{N}}$ is now non-random. The random variables
\[
\xi_n =  d(\0,p_n)   \mathbb{I}_{\Ical_{\bar \theta_N^\epsilon} \Delta \Ical^-_{\theta^*}}(p_n), \quad n = 1,2,\dots,N
\]
are independent and identically distributed with mean and variance
\[
\expE[\xi_n]  = \int_{\Ical_{\bar \theta_N^\epsilon} \Delta \Ical^-_{\theta^*}} d(\0,p) \,d\mu(p), \quad \quad 
\text{Var}(\xi_n) = \expE[\xi_n^2] - \expE[\xi_n]^2 \leq \int_{\Ical_{\bar \theta_N^\epsilon} \Delta \Ical^-_{\theta^*}} d(\0,p)^2 \,d\mu(p).
\]
Due to the square integrability condition (\ref{squareint}), both $\expE[\xi_n]$ and $\text{Var}(\xi_n)$ are finite. Moreover, since
\[
 \Ical_{\bar \theta_N^\epsilon} \Delta \Ical^-_{\theta^*} = \left \{ (r,\theta) \in \KK \;|\; r > 0,\;\; \theta \in (\theta^* - \pi, \bar \theta_N^\epsilon - \pi) \cup [ \bar \theta_N^\epsilon + \pi, \theta^* + \pi) \right\}
\]
we have
\[
\bigcap_{N \geq 1}  \Ical_{\bar \theta_N^\epsilon} \Delta \Ical^-_{\theta^*}  = \emptyset. 
\]
Hence, $\mu(\Ical_{\bar \theta_N^\epsilon} \Delta \Ical^-_{\theta^*})  \to 0$ as $N \to \infty$, and 
\begin{equation} \label{momentvanish}
\lim_{N \to \infty} \int_{\Ical_{\bar \theta_N^\epsilon} \Delta \Ical^-_{\theta^*}} d(\0,p) \,d\mu(p) = 0, \quad \quad \lim_{N \to \infty} \int_{\Ical_{\bar \theta_N^\epsilon} \Delta \Ical^-_{\theta^*}} d(\0,p)^2 \,d\mu(p) = 0.
\end{equation}
Thus, both $\expE[\xi_n]$ and $\text{Var}(\xi_n)$ vanish as $N \to \infty$.  Consequently, for any $\delta > 0$,
\begin{eqnarray}
\limsup_{N \to \infty} \Pm \left( |T_3| \geq \frac{\delta}{\sqrt{N}}, \quad \hat \theta > \theta^* \right) \leq 2\epsilon + \limsup_{N \to \infty} \Pm\left( C_\epsilon \frac{1}{N} \sum_{n=1}^N \xi_n > \delta \right) = 2\epsilon.
\end{eqnarray}
As $\epsilon > 0$ is arbitrary, we conclude that $T_3 = o(1/\sqrt{N})$. The result now follows by combining these estimates of $T_1$, $T_2$, and $T_3$.

Note: the reason we prove that $Z_N$ is $o(1/\sqrt{N})$ rather than a stronger statement like $Z_N \leq O(1/N)$, is that we have no control over the rate at which $\mu(\Ical_{\bar \theta_N^\epsilon} \Delta \Ical^-_{\theta^*})  \to 0$ as $N \to \infty$ or on the rate of convergence in (\ref{momentvanish}), unless we make more assumptions about $\mu$.

\end{proof}

\section{Topological definition of sticky mean}\label{sec:variations}%

\subsection{Topological version for kale}\label{sub:kale}

Let $\mathcal{M}_1$ be the set of all finite Borel measures $\mu$ on
$\KK$ satisfying the integrability condition (\ref{integ}).  This
section considers how the mean (or barycenter) of a measure $\mu \in
\mathcal{M}_1$ varies under perturbations of the measure.  For this
reason, we temporarily modify the notation for $m_{\theta,1}$ to
$m_{\theta,1}(\mu)$, to reflect the measure $\mu$ being used.  It is
then easy to see that for $\mu, \nu \in \mathcal{M}_1$,
\begin{align}\label{m1additive}
  m_{\theta,1}(\mu + \epsilon\nu)
  =
  m_{\theta,1}(\mu) + \epsilon m_{\theta,1}(\nu).
\end{align}
Two measures $\mu, \nu \in \mathcal{M}_1$ are considered equivalent if
they differ only in their total mass, meaning that there is a constant
$c >0$ with $\mu = c \nu$.  Denote the space of equivalence classes by
$\mathcal{\widetilde M}_1$.  Endow $\mathcal{M}_1$ with the topology
generated by the Wasserstein metric defined by
$$%
  \rho(\mu,\nu) = \sup_{f \in\mathrm{Lip}_1} \left( \int f d\mu \,- \int f d\nu \right),
$$
where $\mathrm{Lip}_1$ is the set of real-valued, Lipschitz-continuous
functions on $\KK$ with Lipschitz constant 1.  This topology extends
to $\mathcal{\widetilde M}_1$ by declaring the distance between $\mu$
and~$\nu$ to be the Wasserstein distance $\rho(\mu,\nu)$ when $\mu$
and~$\nu$ are normalized so that $\mu(\KK) = \nu(\KK) = 1$.

Now comes the first in a sequence of results leading us to a
definition of sticky and nonsticky that is more topological than
Definition~\ref{d:stickyclassification}.

\begin{lem}\label{lem:StickyTopDef}
Let $\mu \in \mathcal{\widetilde M}_1$ be fully sticky.  There exists
an open neighborhood $U$ of $\mu$ so that $\nu \in U$ implies
(i)~$\nu$ is fully sticky and (ii)~$\mu$ and $\nu$ have the same mean.
\end{lem}
\begin{proof}
Since the function $e_1 \cdot F_\theta: \KK \to \RR$ is in
$\textrm{Lip}_1$, Lemma~\ref{l:distance-under-folding} yields
\begin{equation}\label{rhom1}
  \sup_\theta |m_{\theta,1}(\mu) - m_{\theta,1}(\nu)| \leq \rho(\mu,\nu)
\end{equation}
for any two measures $\mu, \nu \in \mathcal{\widetilde M}_1$.  If
$\mu$ is fully sticky, then there exists $\epsilon > 0$ so that
$m_{\theta,1}(\mu) \leq - \epsilon < 0$ for all~$\theta$.  Therefore,
if $\rho(\mu,\nu) \leq \epsilon/2$ then $m_{\theta,1}(\nu) \leq
-\epsilon/2<0$ holds for all~$\theta$.  Hence, by
Definition~\ref{d:stickyclassification}, $\nu$ is also fully sticky.
Since all fully sticky measures on the kale~$\KK$ have mean~$\0$, we
conclude that $\mu$ and $\nu$ have the same means.
\end{proof}

\begin{lem}\label{cor:nonStickyOpen}
The set of fully sticky measures is an open subset of
$\mathcal{\widetilde M}_1$, as is the set of nonsticky measures.
\end{lem}
\begin{proof}
The statement for fully sticky measures is contained in
Lemma~\ref{lem:StickyTopDef}.  On the other hand, by
Definition~\ref{d:stickyclassification} the nonsticky measures are
characterized by $m_{\theta,1}$ being strictly positive for an open
range of $\theta$.  Let $\mu$ be a nonsticky measure with
$m_{\theta,1}(\mu) > 2\epsilon$ for $\theta \in (A,B)$, for some
$\epsilon > 0$.  If $\nu \in B_\epsilon(\mu) \subset
\mathcal{\widetilde M}_1$ then (\ref{rhom1}) implies that for all
$\theta \in (A,B)$,
\begin{align*}
  m_{\theta,1}(\nu) \geq \inf_{\theta \in (A,B)} m_{\theta,1}(\mu) -
  \rho(\mu,\nu) > 2\epsilon - \epsilon.
\end{align*}
Therefore all $\nu \in B_\epsilon(\mu)$ are also nonsticky.
\end{proof}

\begin{defn}\label{d:stickyMeasure}
Fix a measure $\mu \in \mathcal{M}_1$.  A measure $\nu \in
\mathcal{M}_1$, thought of as a direction, is
\begin{enumerate}
\item%
\emph{sticky for $\mu$} if $\mu$ and $\mu + \epsilon \nu$ have the
same mean for all sufficiently small $\epsilon > 0$;
\item%
\emph{fluctuating for~$\mu$} if $\mu$ and $\mu + \epsilon \nu$ have
different means for all sufficiently small $\epsilon > 0$.
\end{enumerate}
\end{defn}

Since normalization does not change whether a measure is sticky,
partly sticky, or nonsticky, one could replace $\mu + \epsilon\nu$ by
$(1-\epsilon)\mu + \epsilon\nu$ in the above definition.  The latter
has the advantage of producing a probability measure if both $\mu$ and
$\nu$ were initially so.

It is convenient to have a specific class of perturbations at our
disposal.  Note that for the unit measure $\delta_p$ supported at the
point $p = (r,\theta')$,
\begin{equation}\label{deltam1theta}
  m_{\theta,1}(\delta_p)
  =
  \begin{cases}
  \\[-4.25ex]
    r\cos(\theta - \theta') & \text{if } |\theta - \theta'| < \pi
  \\        -r              & \text{if } |\theta - \theta'| \geq \pi.
  \\[-.9ex]
  \end{cases}
\end{equation}

\begin{lem}\label{lem:nonStickTopDef}
Any nonsticky or partly sticky $\mu \in \mathcal{M}_1$ has a
fluctuating direction in~$\mathcal{M}_1$.
\end{lem}
\begin{proof}
Let $(r^*,\theta^*)$ be the mean of~$\mu$.  When $\mu$ is partly
sticky, $r^* = 0$ and $\theta^*$ is any value; when $\mu$ is
nonsticky, $r^* = m_{\theta^*,1} > 0$ and $\theta \mapsto
m_{\theta,1}(\mu)$ attains its maximum at the unique point~$\theta^*$.
Fix any radius $r > 0$ with $r \neq r^*$, and set $\mu_\epsilon =
(1-\epsilon)\mu + \epsilon \delta_{(r,\theta^*)}$.  By
(\ref{m1additive}) and~(\ref{deltam1theta}), $\theta \mapsto
m_{\theta,1}(\mu_\epsilon)$ now has its unique maximum at~$\theta^*$, but
$m_{\theta^*,1}(\mu_\epsilon) \neq m_{\theta^*,1}(\mu)$ because $r
\neq r^*$.  Hence $\mu$ and $\mu_\epsilon$ have different means, so
the direction $\delta_{(r,\theta^*)}$ is fluctuating for~$\mu$.
\end{proof}

\begin{lem}\label{lem:partStickyNearSticky}
If $\mu\in \mathcal{M}_1$ is partly sticky then $\mu$ has a sticky
direction (other than $\nu = \mu$).
\end{lem}
\begin{proof}
Since $\mu$ is partly sticky, $m_{\theta,1}(\mu) \leq 0$ for
all~$\theta$.  Let $\nu$ be any fully sticky measure and define
$\mu_\epsilon = (1-\epsilon)\mu + \epsilon\nu$.  Since $\nu$ is fully
sticky, $m_{\theta,1}(\nu) < 0$ for all~$\theta$, and hence
$m_{\theta,1}(\mu_\epsilon) < 0$ for all~$\theta$ as long as $\epsilon
> 0$.  Therefore $\mu_\epsilon$ is fully sticky, and the means
of~$\mu_\epsilon$ and~$\mu$ coincide at~$\0$ for all $\epsilon \in
[0,1]$.  Thus $\nu \neq \mu$ is a sticky direction for~$\mu$.
\end{proof}

The above lemmas combine with the fact that all measures in
$\mathcal{\widetilde M}_1$ are either fully sticky, partly sticky, or
nonsticky (Proposition~\ref{p:cases} and
Definition~\ref{d:stickyclassification}) to prove the following
theorem, which could be seen as an alternative definition of the terms
``fully sticky'', ``partly sticky'', and ``nonsticky'' for finite
measures on~$\KK$.

\begin{thm}\label{t:TopSticky}
Let $\Scal \subset \mathcal{\widetilde M}_1$ be the open subset of
fully sticky measures.  A measure $\mu \in \mathcal{\widetilde
M}_1$~is
\begin{enumerate}
\item%
fully sticky (i.e.\ $\mu \in \Scal$) if and only if there is an open
neighborhood of $\mu$ so that all measures in that neighborhood have
the same mean as $\mu$.  Equivalently, a measure $\mu$ is fully sticky
if and only if all directions $\nu \in \mathcal{\widetilde M}_1$ are
sticky for $\mu$.
\item%
partly sticky if and only if $\mu \in \partial \Scal$, the topological
boundary of $\Scal$.  Equivalently, a measure $\mu$ is partly sticky
if and only if every open neighborhood of $\mu$ contains open sets $U$
and $V$ such that $\nu \in V \implies \nu$ has the same mean as~$\mu$
and $\nu \in U \implies \mu$ and~$\nu$ have different means.
\item%
nonsticky if and only if $\mu \in \mathcal{\widetilde M}_1 \setminus
\ol{\Scal}$, the compliment of the closure of~$\Scal$.  Equivalently,
a measure $\mu$ is nonsticky if and only if no open neighborhood
of~$\mu$ contains an open set $U$ consisting of measures with the same
mean as~$\mu$.
\end{enumerate}
\end{thm}

\begin{rmk} As $N$ gets large, the empirical measure
$$%
  \mu_N = \frac{1}{N} \sum_{n=1}^N \delta_{p_n}
$$
converges to $\mu$ in the topology generated by $\rho$ if the $p_n$
are chosen independently and according to $\mu$.  (For instance
combine \cite[Theorem~6.9]{villani09} and the standard weak
convergence of empirical measures.)  If $\mu$ is sticky then
eventually $\mu_N$ lies in a neighborhood of~$\mu$ in which all
measures have the same mean.  On the other hand, if $\mu$ is
nonsticky then nearby measures have different means than $\mu$ and
hence the mean of $\mu_N$ fluctuates.  When $\mu$ is partly sticky,
sometimes $\mu_N$ lies in a set of measures sharing their mean
with~$\mu$, and sometimes it lies in a set of measures having
different means than~$\mu$.
\end{rmk}

\begin{rmk}\label{r:topology}
Endowing $\mathcal{M}_1$ instead with the topology generated by the
open neighborhoods
$$%
  U_{\mu,\epsilon}
  =
  \{\nu\in\mathcal{M}_1 \,\big| \;\;\max_\theta |m_{\theta,1}(\mu) -
  m_{\theta,1}(\nu)| < \epsilon\}
$$
maintains the truth of the above results.  However, using the standard
weak topology on measures, which is finer, would cause the topological
characterization of stickiness to fail.
\end{rmk}

\subsection{Topological definition for arbitrary metric spaces}\label{sub:arbitrary}

Suppose $\KK$ is a metric space, and let $\Mcal$ be a set of
probability measures on $\KK$.

\begin{ex}\label{e:topology}
When $\Mcal = \widetilde\Mcal_1$ is the set of Borel probability
measures on~$\KK$ satisfying the integrability
condition~(\ref{integ}), different topologies on~$\Mcal$ are induced
by the Wasserstein metric and by the sets $U_{\mu,\epsilon}$ in
Remark~\ref{r:topology}.  The standard weak topology is yet another
possibility.
\end{ex}

\begin{defn}\label{d:sticky}
Let $\Mcal$ be a set of measures on a metric space~$\KK$ with the metric topology.  Assume
$\Mcal$ has a given topology.  A \emph{mean} is a continuous
assignment $\Mcal \to \{$closed subsets of~$\KK\}$.  A measure~$\mu$
\emph{sticks} to a closed subset $C \subseteq \KK$ if every
neighborhood of~$\mu$ in~$\Mcal$ contains a nonempty open subset consisting of
measures whose mean sets are contained in~$C$.
\end{defn}

\begin{rmk}
Regarding the topology on the set of closed subsets of $\KK$, implicit in Definition \ref{d:sticky}, we have in mind the topology induced by the \emph{Hausdorff distance}:
\[
d(A,B) = \max \left\{  \sup_{a \in A} d(a,B)\;,\; \sup_{b \in B} d(b,A)  \right\}.
\] 
That is, $d(A,B)$ is the farthest a point of $A$ is from $B$ or the farthest a point of $B$ is from $A$, whichever is greater. Other topologies on the set of closed subsets of $\KK$ are possible, such as the ``pointed Hausdorff topology", which is compact and locally compact.
\end{rmk}

Continuity implies that the mean of $\mu$ is contained in~$C$ if~$\mu$
sticks to~$C$.

\begin{ex}\label{e:kale}
This paper has investigated measures on the kale~$\KK$, which can
stick to the subset $C = \{\0\}$ consisting of the origin.  The notion
of ``mean'' here is Definition~\ref{d:meanK}, which assigns to each
measure a single point; this assignment is continuous by
Lemma~\ref{lem:Dmtheta}.
\end{ex}

In spaces of interest, integrability conditions, such as those in
Section~\ref{s:results} here, would imply existence of means.
However, means in general metric spaces---even nice ones such as
compact Riemannian manifolds---need not be single points.  In other
words, the general analogue of the minimization problem in
Section~\ref{sub:barycenters} could have multiple solutions.  For
instance the mean set of the uniform measure on a sphere is equal to
that entire sphere, whereas each sample mean is unique almost surely
(cf.\ Remark~2.6 in \cite{BP03}).  In Section~5 of [\cite{HH14}] there
is an example of a measure on the circle where the mean set is a
proper circular arc.  In fact, this can be viewed as the limiting case
of measures with unique means, the central limit theorems for which
feature arbitrarily slow convergence rates.  Uniqueness of means for
the kale stem from its negative curvature; see
\cite[Proposition~4.3]{Sturm2003}, for example.

\begin{rmk}\label{r:sticky}
In the language of earlier sections, Definition~\ref{d:sticky} only
sets forth the notion of ``sticky'', which includes both the sticky
and partly sticky cases.  In the generality of
Definition~\ref{d:sticky}, it would be said that a measure~$\mu$
\emph{fully sticks to~$C$} if some open neighborhood of~$\mu$ consists
entirely of measures whose means are contained within~$C$.  It would
not be required that the means (closed subsets of $C$) of the measures in such a neighborhood
should equal the mean of $\mu$ or even intersect it at all.  In the case
where~$\KK$ is an open book [\cite{HHMMN13}], for example, means are
unique and measures can stick to the spine, but nothing prevents the
mean of a sticky measure from moving along the spine.

The set of partly sticky measures would be defined as those that are
sticky but not fully sticky.  Definition~\ref{d:sticky} implies that
the set of partly sticky measures is the topological boundary of the
set of sticky measures.
\end{rmk}

It remains open to characterize which metric spaces---among, say, the
topologically stratified spaces (see [\cite{GM1988}] or
[\cite{Pflaum2001}]), to be concrete---admit measures that stick to
subsets of measure~$0$.  Given such a sticky situation, first goals
would be to prove laws of large numbers and central limit theorems,
contrasting the fully, partly, and nonsticky cases.  The limiting
measures in such results would be singular analogues of Gaussian
distributions; it is not clear what properties of Gaussian
distributions are the right ones to lift so as to characterize the
building blocks of limiting measures in general.

\newpage 
\section{List of Notation} \label{sec:notation}

\begin{tabular}{cl}
$d(p,q)$ & Metric on $\KK$. See Section \ref{sub:isolated}.\\
$F_\theta$ & The folding map, from $\KK$ to $\Rm^d$, at angle $\theta$. Definition \ref{d:folding}. \\
$\mathcal{I}_\theta $ & The shadow of angle $\theta$; an open subset of $\KK$. Definition \ref{d:IcalA}. \\
$\mu$ & A probability measure on $\KK$.\\
$\mu_N$ & The empirical measure for points  $p_1,\dots,p_N \in \KK$. See Section \ref{sub:empirical}. \\
$b_N$ & The barycenter of a (random) set of points $p_1,\dots,p_N \in \KK$. See (\ref{d:sample-barycenter-kale}).\\
$\bar b$ & Population barycenter. See Definition \ref{d:populationbarycenter}.\\
$\tilde \mu_\theta$ & The pushforward $\mu \circ F_\theta^{-1}$ of $\mu$ under $F_\theta$; a measure on $\Rm^2$. \\
$m_\theta $ & First moment of measure $\mu$ folded about angle $\theta$. Definition \ref{d:moments}.\\
$m_\theta^N$ & First moment of the empirical measure $\mu_N$ folded about angle $\theta$. Definition \ref{d:moments}.\\
$m_{\theta,1}$, $m_{\theta,2}$ & Components of $m_\theta = (m_{\theta,1}, m_{\theta,2}) \in \Rm^2$.\\
$\eta_{\theta,N}$ & Folded average, equivalent to $m_{\theta}^N$. See (\ref{etaNMequiv}).\\
$\nu_N$ & Distribution of rescaled empirical means, a probability measure on $\KK$. See (\ref{eq:rescaled-emp-means}).\\
$\kappa(\omega)$ & A random variable related to the CLT in the non-sticky case. See (\ref{eq:kappa}). \\
$w^\pm(\theta)$ & See (\ref{eq:kappa}). \\
$\mathcal{I}^\pm_\theta$ & Shadow at angle $\theta$ including part of the boundary. See (\ref{wpmdef}).\\
$\Phi_\sigma$ & Rotation in $\Rm^2$ by angle $\sigma$. See Lemma \ref{lem:mthetamodulus}.  \\
$\bar r$ & Constant bounding first moments of the measure $\mu$. See (\ref{integ}). \\
$\hat P_\rho$ & Convex projection onto a sector in $\Rm^2$. See (\ref{eq:convex projection}).\\
$g$ & Gaussian measure on $\Rm^2$ with mean zero, covariance $\Sigma$. See Sections \ref{sub:partly} and \ref{sub:non-sticky}.
\end{tabular}



\end{document}